\documentclass[10pt, reqno]{amsart}

\usepackage{amssymb, amsmath, amsthm}
\usepackage{enumitem}
\usepackage[all]{xy}
\usepackage[normalem]{ulem}
\usepackage{stmaryrd}
\usepackage{mathtools}
\usepackage{verbatim}
\usepackage{extarrows}
\usepackage{stackrel}
\usepackage{tikz}
\usepackage{mathrsfs}
\usepackage[french, english]{babel}
\usepackage[normalem]{ulem}
\usepackage{xcolor}

\colorlet{lightlightgray}{gray!20}
\colorlet{light1.5gray}{gray!35}

\usepackage[colorlinks=true, pdfstartview=FitV, linkcolor=blue, citecolor=blue, urlcolor=blue]{hyperref}



\newtheorem{theorem}{Theorem}[section]
\newtheorem*{theorem*}{Theorem}
\newtheorem{maintheorem}{Theorem}
\newtheorem{maincorollary}[maintheorem]{Corollary}
\newtheorem*{maintheorem*}{Theorem}
\newtheorem{lemma}[theorem]{Lemma}
\newtheorem{corollary}[theorem]{Corollary}
\newtheorem{proposition}[theorem]{Proposition}

\newtheorem*{question*}{Question}

\theoremstyle{remark}
\newtheorem{remark}[theorem]{Remark}
\newtheorem*{remark*}{Remark}

\theoremstyle{definition}
\newtheorem{definition}[theorem]{Definition}
\newtheorem{claim}{Claim}

\setlist[1]{labelindent=\parindent, leftmargin=*}

\DeclareMathOperator{\GL}{GL}

\DeclareMathOperator{\SL}{SL}

\DeclareMathOperator{\Aut}{Aut}
\DeclareMathOperator{\End}{End}

\DeclareMathOperator{\Cent}{Cent}
\DeclareMathOperator{\Bir}{Bir}
\DeclareMathOperator{\id}{id}

\DeclareMathOperator{\Spec}{Spec}

\DeclareMathOperator{\Hom}{Hom}

\DeclareMathOperator{\rank}{rank}

\DeclareMathOperator{\Span}{Span}
\DeclareMathOperator{\Mor}{Mor}

\DeclareMathOperator{\grp}{grp}

\DeclareMathOperator{\dom}{dom}

\DeclareMathOperator{\A}{R}

\DeclareMathOperator{\Repl}{Repl}
\DeclareMathOperator{\rRepl}{rRepl}

\newcommand{\GG}{\mathbb{G}}
\newcommand{\RR}{\mathbb{R}}
\newcommand{\CC}{\mathbb{C}}
\newcommand{\QQ}{\mathbb{Q}}
\newcommand{\PP}{\mathbb{P}}
\newcommand{\ZZ}{\mathbb{Z}}
\newcommand{\NN}{\mathbb{N}}
\renewcommand{\AA}{\mathbb{A}}

\newcommand{\kk}{\textbf{k}}

\newcommand{\aquot}{/ \! \! /}

\newcommand{\name}[1]{#1}

\renewcommand{\phi}{\varphi}

\newcommand{\set}[2]{\left\{\,#1 \ | \ #2\,\right\}}
\newcommand{\Bigset}[2]{\left\{\,#1 \ \Big| \ #2\,\right\}}
\newcommand{\sprod}[2]{\langle #1, #2 \rangle}

\title[Maximal commutative unipotent subgroups]
{Maximal commutative unipotent subgroups and a characterization of affine spherical varieties}
\author[A. Regeta \and I. van Santen]
{Andriy Regeta \and Immanuel van Santen}

\thanks{}

\address{\noindent Institut f\"{u}r Mathematik, Friedrich-Schiller-Universit\"{a}t Jena, \newline
	\indent  Jena 07737, Germany}
\email{andriyregeta@gmail.com}

\address{Departement Mathematik und Informatik, 
	Universit\"at Basel,\newline
	\indent Spiegelgasse 1, CH-4051 Basel, Switzerland}
\email{immanuel.van.santen@math.ch}

\begin{document}
	

\setcounter{tocdepth}{1}

\begin{abstract}
%

We describe maximal commutative unipotent subgroups of the automorphism group $\mathrm{Aut}(X)$ of an irreducible affine variety $X$. Further we show that
a group isomorphism $\mathrm{Aut}(X) \to \mathrm{Aut}(Y)$ maps unipotent elements to unipotent elements,
where $Y$ is irreducible and affine.
Using this result, we show that the automorphism group detects sphericity and the weight-monoid.

As an application, we show that an affine toric variety different from an algebraic torus is determined by its automorphism group among normal irreducible affine varieties and we show that a smooth affine spherical variety different from an algebraic torus is determined by its automorphism group 
(up to an automorphism of the base field) among smooth irreducible affine varieties.
This generalizes results obtained by Cantat, Kraft, Liendo, Urech, Xie and the authors.
\end{abstract}

\maketitle
	
\tableofcontents	

\section{Introduction}

We work over an algebraically closed field $\kk$ of characteristic zero. Let $X$ be an irreducible affine variety.
In this paper we study subgroups of the automorphism group $\Aut(X)$ of $X$.
As an application we study the question of characterizing affine spherical varieties via 
their automorphism groups. This question is well-studied in the context of differentiable manifold theory, 
see e.g.~\cite{Wh1963On-isomorphic-grou,Fi1982Isomorphisms-betwe,Ry1995Isomorphisms-betwe, 
	Ry2002Isomorphisms-betwe}.

In \cite{Sh1966On-some-infinite-d}  \name{Shafarevich} introduced the notion of 
an \emph{ind-group}, which is an infinite-dimensional analogue of an algebraic group.
It is known that the automorphism group  $\Aut(X)$ has a natural structure of an 
ind-group (see \cite[Section 5]{FuKr2018On-the-geometry-of} or Sect.~\ref{Sec.Preliminaries} for details). 
In particular $\Aut(X)$
has a natural topology induced by this ind-group structure.
Through the whole paper we consider only \emph{linear} algebraic groups and therefore we will drop
the adjective ``linear''.

\medskip

An algebraic group $G$ consists only of unipotent elements if and only if the neutral element is the 
only finite order element in $G$ (one can use the Jordan decomposition of an element, see e.g.~\cite[Theorem~5.3]{Hu1975Linear-algebraic-g}). This statement holds as well for closed connected commutative
subgroups $G$ in $\Aut(X)$, where $\varphi \in \Aut(X)$ is called \emph{unipotent} if there is an 
algebraic action of the one-dimensional additive group $\GG_a$ on $X$ such that $\varphi$ lies in the image of 
the natural induced homomorphism $\GG_a \to \Aut(X)$ 
(see \cite[Theorem~B]{CaReXi2023Families-of-commut} and Corollary~\ref{cor.CRX}).
%
%
We call a subgroup of $\Aut(X)$ \emph{unipotent} if it is closed and consists only of unipotent elements,
and we call it \emph{algebraic} if it is the image under the natural homomorphism induced by an algebraic action 
of an algebraic group on $X$. Note that unipotent subgroups of $\Aut(X)$ are always connected
and algebraic subgroups of $\Aut(X)$ are always closed.

\medskip

In this paper we describe commutative unipotent subgroups of 
$\Aut(X)$ that are maximal among such subgroups. This could be seen as a first step 
towards a systematic study of  maximal unipotent (or solvable) subgroups  of $\Aut(X)$.

In order to state our first result we have to introduce some notation.
Let $U \subseteq \Aut(X)$ be a unipotent algebraic subgroup and let $u \in U \setminus \{ \id_X\}$.
Let $\rho \colon \GG_a \times X \to X$ be a $\GG_a$-action on $X$ such that $\rho(1, x) = u(x)$ 
for all $x \in X$. For a rational $U$-invariant function $f \in \kk(X)^U$ denote by
$\dom(f) \subseteq X$ its domain. Then
\[
f \cdot u \colon \dom(f) \to \dom(f) \, , \quad x \mapsto \rho(f(x), x) 
\]
defines an automorphism of $\dom(f)$ and hence $f \cdot u$ is an element of
the group of birational transformations $\Bir(X)$ of $X$. We define
\begin{eqnarray*}
	\Repl_X(U) &\coloneqq& \set{ f \cdot u \in \Aut(X) }{ f \in \kk[X]^U, u \in U} \subseteq \Aut(X) \, , \\
	\rRepl_X(U) &\coloneqq& \set{ f \cdot u \in \Bir(X) }{ f \in \kk(X)^U, u \in U}  \subseteq \Bir(X) \, .
\end{eqnarray*}
The term $``\Repl"$ stands for ``replica'' (see e.g.~\cite{ArFlKa2013Flexible-varieties}) and $``\rRepl"$ stands for 
``rational replica''.
Denote by $\langle \rRepl_X(U) \rangle$ the subgroup
generated by $\rRepl_X(U)$ inside $\Bir(X)$ and let
\[
\A_X(U) \coloneqq \Aut(X) \cap \langle \rRepl_X(U) \rangle \subseteq \Bir(X) \, .
\]

\begin{maintheorem}[{see Corollary~\ref{Cor.Maximal2}}]
	\label{mainthm.max}
	Assume $X$ is an irreducible affine variety and let $G$ be a unipotent 
	commutative subgroup in $\Aut(X)$. 
	Then the following statements are equivalent:
	\begin{enumerate}[wide=0pt, label=\alph*), leftmargin=*]
		\item \label{mainthm.maxA} $G$  is maximal among commutative subgroups in $\Aut(X)$;
		\item \label{mainthm.maxB} $G$ is maximal among unipotent commutative subgroups in $\Aut(X)$;
		\item \label{mainthm.maxC} $G = \A_X(U)$ 
		for some  commutative unipotent algebraic subgroup $U \subseteq \Aut(X)$.
	\end{enumerate}
\end{maintheorem}

The proof of Theorem~\ref{mainthm.max} depends heavily on the fact that the centralizer
of $\Repl_X(U)$ is equal to $\A_X(U)$ which we prove in Proposition~\ref{Prop.centralizer}.
As a consequence of Theorem~\ref{mainthm.max} we get:

\begin{maincorollary}
	[see Corollary~\ref{Cor.UnipotentToUnipotent}]	
	\label{algebraictoalgebraic}
	Assume $\kk$ is uncountable. Let $X, Y$ be irreducible affine varieties and let
	$\theta \colon \Aut(X) \to \Aut(Y)$ be a group isomorphism. Then 
	$\theta$ maps unipotent elements to unipotent elements.
\end{maincorollary}

\medskip

Corollary~\ref{algebraictoalgebraic} provides an important tool to study the following question when the objects under consideration are affine varieties.

\begin{question*}
	Let $X$, $Y$ be objects in some category $\mathscr{C}$. 
	If there exists a group isomorphism
	$\theta \colon \Aut_{\mathscr{C}}(X) \to \Aut_{\mathscr{C}}(Y)$, 
	are then $X$, $Y$ isomorphic in $\mathscr{C}$?
\end{question*}

Of course, one cannot hope for an affirmative answer of this question in the whole 
category of varieties, as the automorphism groups of varieties are trivial in general. 
However, if $X$ is rather special, over the field of complex numbers 
we have affirmative answers in the following instances
(when the morphisms in the category $\mathscr{C}$ are not specified, we consider morphisms of varieties;
some of the results hold over more general fields than the complex numbers):
\begin{enumerate}[wide=0pt, label=\roman*), leftmargin=*]
	\item \label{Item.History_1} $X = \PP^n$ is the projective space, where
	$\mathscr{C}$ is the category of irreducible projective $n$-dimensional varieties together
	with rational maps; see~\cite[Theorem~C]{Ca2014Morphisms-between-};
	\item \label{Item.History_2} $X = \AA^n$ is the affine space, where $n \geq 1$,
	$\mathscr{C}$ is the category of connected affine varieties 
	and $\theta$ is an ind-group homomorphism; see~\cite[Theorem~1.1]{Kr2017Automorphism-group};
	\item \label{Item.History_3} 
	$X$ is a Danielewski surface, where
	$\mathscr{C}$ is the category of irreducible smooth affine surfaces and $\theta$ is an ind-group isomorphism; see \cite[Theorem~1 and Remark~1]{LiReUr2022On-the-Characteriz};						
	\item \label{Item.History_3.5} $X$ is a Danielewski surface, where
	$\mathscr{C}$ is the category of irreducible affine varieties and $\theta$ is an ind-group isomorphism; 
	see \cite[Theorem~2]{Re2022Characterization-o};
	\item \label{Item.History_5} 
	$X$ is an affine toric variety that is non-isomorphic to an algebraic torus, where
	$\mathscr{C}$ is the category of irreducible normal affine varieties
	and $\theta$ is an ind-group isomorphism;
	see \cite[Theorem~1.4]{LiReUr2023Characterization-o};
	\item \label{Item.History_4} $X$ is a smooth affine spherical variety that is
	non-isomorphic to an algebraic torus, 
	where $\mathscr{C}$ is the category of irreducible smooth affine varieties 
	and $\theta$ is an ind-group isomorphism; 
	see~\cite[Theorem~A]{ReSa2021Characterizing-smo};
	\item \label{Item.History_6}
	$X = \AA^n$, where
	$\mathscr{C}$ is the category of smooth irreducible quasi-projective varieties of dimension $n$
	such that the Euler characteristic doesn't vanish and the Picard group is finite;
	see~\cite[Theorem]{KrReSa2021Is-the-affine-spac};	
	\item \label{Item.History_7}
	$X$ is an affine toric surface, where
	$\mathscr{C}$ is the category of irreducible normal affine surfaces;
	see \cite[Theorem~1.3]{LiReUr2023Characterization-o};
	\item \label{Item.History_8}
	$X = \AA^n$, where $n \geq 1$ and $\mathscr{C}$ is the category of 
	connected affine varieties; see \cite[Theorem~A]{CaReXi2023Families-of-commut}.
\end{enumerate}
Note that the assumption that $\Aut(X)$, $\Aut(Y)$ are isomorphic as abstract groups is much weaker than the assumption that these groups are isomorphic as ind-groups
(as it is assumed in cases~\ref{Item.History_2}-\ref{Item.History_4}). 
Indeed, e.g.~the automorphism groups of 
two generic Danielewski surfaces are isomorphic as abstract groups, 
but non-isomorphic as ind-groups; see \cite[Theorem 3 and Remark~7]{LeRe2022Vector-fields-and-}).

In Theorem~\ref{mainthm.A} and Corollary~\ref{maincor.B} 
below we generalize statements~\ref{Item.History_2}, \ref{Item.History_5}-\ref{Item.History_8}
to affine smooth spherical (affine toric) varieties
when $\mathscr{C}$ is the category of smooth (normal) irreducible affine varieties.


Let $X$
be an irreducible affine variety endowed with a faithful algebraic group action of
a connected reductive algebraic group $G$ 
and let $B \subseteq G$ be a Borel subgroup. Then $X$ is called \emph{$G$-spherical} if
$X$ is normal and $B$ acts with a dense open orbit on $X$. 
We denote by  $\frak{X}(B)$ the character group of $B$, i.e. the group of
algebraic group homomorphisms $B \to \mathbb{G}_m$.
The \emph{weight monoid} of $X$ is defined by
\[
\Lambda_B^+(X) \coloneqq \{ \,\lambda \in \frak{X}(B) \ | \ \kk[X]_{\lambda}^{(B)} \neq 0 \, \} \, ,
\] 
where $\kk[X]_{\lambda}^{(B)} \subset \kk[X]$ 
denotes the subspace of $B$-semi-invariants of weight $\lambda$
of the coordinate ring $\kk[X]$.

In order to simplify the formulation, we state Theorem~\ref{mainthm.A}
only when $X$  is defined over $\QQ$ (we give a more general version in Theorem~\ref{thm.Ageneralization}). 
Assume that $X$ is a closed subvariety of some affine space $\AA^n$ and assume that
$X$ is defined over $\QQ$, i.e.~$X$ is the zero set in $\AA^n$ of some polynomials with coefficients in $\QQ$.
If $\tau$ is a field automorphism of
$\kk$, then the map $\AA^n \to \AA^n$, $(a_1, \ldots, a_n) \mapsto (\tau(a_1), \ldots, \tau(a_n))$
sends $X$ onto itself and induces thus a map
\[
	\tau_X \colon X \to X 
\]
(in fact $\tau_X$ only depends on $X$, not on the embedding of $X$ into $\AA^n$).
If $\theta \colon \Aut(X) \to \Aut(Y)$ is a group isomorphism, then 
\[
	\Aut(X) \to \Aut(Y) \, , \quad \varphi \mapsto \theta( \tau_X^{-1} \circ \varphi \circ \tau_X)
\]
is again a group isomorphism. This is the reason why field automorphisms
of $\kk$ appear in the next theorem:

\begin{maintheorem}\label{mainthm.A}
	Assume that $\kk$ is uncountable, let $X$, $Y$ be irreducible affine varieties, where $X$ is
	not isomorphic to an algebraic torus, let
	$\theta \colon \Aut(X) \to \Aut(Y)$ be a group isomorphism
	and let $G \subseteq \Aut(X)$ be a connected reductive 
	algebraic subgroup.
	If $X$ is $G$-spherical and if $X$, $G$ and the action of $G$ on $X$ are 
	defined over $\QQ$, then there
	exists a  field automorphism $\tau$ of $\kk$ such that:
	\begin{enumerate}[wide=0pt, leftmargin=*]
		\item
		The image $\theta(G)$ is an algebraic subgroup of $\Aut(Y)$ and
		$\theta |_G \circ \tau_G \colon G \to \theta(G)$ is an isomorphism of algebraic groups. 
	\end{enumerate}
	For the following statements we endow $Y$ with the $G$-action induced by $\theta |_G \circ \tau_G$
	and assume in addition that $Y$ is normal.
	\begin{enumerate}[wide=0pt, leftmargin=*]
		\addtocounter{enumi}{1}
		\item The variety $Y$ is $G$-spherical.
		\item
		If $B \subseteq G$ is a Borel subgroup defined over $\QQ$, then
		the weight monoids $\Lambda^+_{B}(X)$  and $\Lambda^+_{B}(Y)$
		are the same inside $\frak{X}(B)$.
		\item If $X$ and $Y$ are smooth, then they are $G$-equivariantly isomorphic.
	\end{enumerate}
\end{maintheorem}

\begin{remark}
	In fact (up to isomorphism) every connected reductive algebraic group is defined over $\QQ$ 
	by \cite[Exp. XXV, Corollaire 1.3 and Exp. XXII, Corollaire~2.4]{DeGr2011Schemas-en-groupes}.
	To the authors' knowledge, it is not known, whether every (smooth) affine spherical variety is defined over $\QQ$.
\end{remark}

The proof of Theorem~\ref{mainthm.A} is given at the end of Sect.~\ref{Sec.CharSpherical_different_from_torus}. As we already mentioned, it depends on
the fact that unipotent elements are mapped onto unipotent elements via
the group isomorphism $\theta \colon \Aut(X) \to \Aut(Y)$. Moreover, we also prove that the
root subgroups of $\Aut(X)$ correspond bijectively to the root subgroups of $\Aut(Y)$ via $\theta$
(see Theorem~\ref{Prop.Preserving_alg_groups}). This 
is another major ingredient in the proof of Theorem~\ref{mainthm.A}.

For the following consequence of Theorem~\ref{mainthm.A}, recall that a 
toric variety is a $G$-spherical variety, where $G$ is an algebraic torus.
We prove it at the end of Sect.~\ref{Sec.CharSpherical_different_from_torus}.

\begin{maincorollary}
	\label{maincor.B}
	Assume that $\kk$ is uncountable, let $X, Y$ be irreducible normal
	affine varieties such that $\Aut(X)$ and $\Aut(Y)$ are isomorphic as abstract groups.
	If $X$ is toric and non-isomorphic to an algebraic torus, then $Y$ is isomorphic to $X$. \qed
\end{maincorollary}

\begin{remark*}[see Remark~\ref{Rem.Both toric}]
	In case $X, Y$ are both affine toric varieties and $\Aut(X)$, $\Aut(Y)$ are isomorphic as groups,
	then $X, Y$ are isomorphic as toric varieties. 
\end{remark*}

Corollary~\ref{maincor.B} fails 
if $X$ is an algebraic torus: 
For every irreducible smooth affine curve $C$ with trivial $\Aut(C)$ 
that admits no non-constant
invertible regular function $C \to \AA^1$ we get that
$\Aut(X)$ and $\Aut(X \times C)$ are isomorphic by \cite[Example~6.17]{LiReUr2023Characterization-o}. 
However, if we assume that
$\dim Y \leq \dim X$, then
Corollary~\ref{maincor.B} holds for an algebraic torus $X$ as well (the proof is given
at the end of Sec.~\ref{sec.X_is_the_torus}):

\begin{maintheorem}
	\label{mainthm.C}
	Assume that $\kk$ is uncountable.
	Let $X$ be an algebraic torus and let
	$Y$ be an irreducible normal affine variety such that
	$\Aut(X), \Aut(Y)$ are isomorphic as groups and $\dim Y \leq \dim X$.
	Then $X$, $Y$ are isomorphic as varieties.
\end{maintheorem}	

If we replace in Corollary~\ref{maincor.B} the assumption that $X$, $Y$
are affine by the assumption that $X$, $Y$ are projective, then Corollary~\ref{maincor.B} fails. 
We give counterexamples in Sec.~\ref{sec.Counterexample_proj_case}.
Moreover, the normality assumption in Corollary~\ref{maincor.B}
is also necessary: There exists an irreducible non-normal 
affine variety $Y$ such that its normalization $\tilde{Y}$ is an affine 
toric variety that is non-isomorphic to an algebraic torus
and the normalization map $\tilde{Y} \to Y$ induces a group isomorphism $\Aut(Y) \to \Aut(\tilde{Y})$, see~\cite[Proposition~9.1]{ReSa2021Characterizing-smo} and \cite[Theorem~4.5]{DiLi2023On-the-Automorphis}. 
Also the irreducibility of
$Y$ in Corollary~\ref{maincor.B} is needed: In fact, if $X$ is any 
affine toric variety of positive dimension and if $Y$ is the disjoint union of $X$
with any affine normal variety that has a trivial automorphism group, then $\Aut(X)$ and $\Aut(Y)$
are isomorphic.

\subsection*{Acknowledgements}
The authors would like to thank \name{J\'er\'emy Blanc}, \name{Michel Brion}, \name{Peter Feller},
\name{Alexander Perepechko} and
\name{Mikhail Zaidenberg} 
for fruitful discussions and comments. We also thank \name{Ivan Arzhantsev} who pointed out to us
the proof of Lemma~\ref{Lem.No_B-root subgroups}.

\section{Preliminary results on ind-groups and algebraic subgroups}
\label{Sec.Preliminaries}

We refer to the survey~\cite{FuKr2018On-the-geometry-of} for the basic notions and results on ind-varieties and ind-groups. An  \emph{ind-variety} $Z$ is a set together with a countable
increasing filtration
of closed subvarieties
\[
Z_1 \subseteq Z_2 \subseteq Z_3 \subseteq \cdots \subseteq Z
\]
such that $Z = \bigcup_{i=1}^\infty Z_i$. 
The ind-variety $Z$ is called \emph{affine} if $Z_i$ is affine for all $i \geq 1$.
Declaring a subset $A \subseteq Z$ to be \emph{closed},
if $A \cap Z_i$ is closed in $Z_i$ for all $i \geq 1$, turns $Z$ into a topological space. Moreover, 
a subset $S \subseteq Z$ is called a \emph{closed subvariety} of $Z$ if 
there exists $i \geq 1$ such that $S \subseteq Z_i$ and $S$ is a closed subvariety in $Z_i$.
A \emph{morphism of ind-varieties} is a map
\[
f \colon Z = \bigcup_{i=1}^\infty Z_i \to W = \bigcup_{j=1}^\infty W_j
\]
such that for all $i \geq 1$ there exists $j = j(i) \geq 1$ with $f(Z_i) \subseteq W_j$ and
$f |_{Z_i} \colon Z_i \to W_j$ is a morphism of varieties. 

An \emph{ind-group} is a group $G$ that is at the same time an affine 
ind-variety and the group operations
are morphisms of ind-varieties. We denote by $G^\circ$ the connected component of an ind-variety $G$
that contains the neutral element. It is a countable index subgroup in $G$.
An \emph{ind-group homomorphism} $G \to H$
is a homomorphism of groups that is at the same time a morphism of ind-varieties.

If $X$ is an affine variety, then $\Aut(X)$ admits in a natural way the structure of an ind-group
such that for all algebraic groups $G$ the  
algebraic group actions of $G$ on $X$ correspond to the ind-group
homomorphisms $G \to \Aut(X)$. The images of the ind-group homomorphisms
$G \to \Aut(X)$ are called \emph{algebraic subgroups} of $\Aut(X)$.
It then turns out that a subgroup in $\Aut(X)$ is algebraic if and only
if it is a closed subvariety in $\Aut(X)$.
For a subset $S \subseteq \Aut(X)$, we denote by $\langle S \rangle \subseteq \Aut(X)$
the subgroup of $\Aut(X)$ that is generated by $S$.

All our theorems heavily depend on the following result due to Cantat,  Xie and the first author:

\begin{theorem}[{see \cite[Theorem~B]{CaReXi2023Families-of-commut}}]
	\label{thm.CRX}
	Let $X$ be an affine variety and  let $V$ be an irreducible closed subvariety of $\Aut(X)$
	such that 
	\[
	\id_X \in V \quad \textrm{and} \quad 
	\varphi \circ \psi = \psi \circ \varphi \quad
	\textrm{for all $\varphi, \psi \in V$} \, .
	\]
	Then $\langle V \rangle \subseteq \Aut(X)$ is an algebraic subgroup of $\Aut(X)$. \qed
\end{theorem}

As a consequence of this theorem we may write closed commutative connected subgroups
in the automorphism group of an affine variety in a very nice form:

\begin{corollary}
	[{\cite[Remark 1.1]{CaReXi2023Families-of-commut}}]\label{cor.CRX}
	Let $X$ be an affine variety and let 
	$G$ be a closed commutative connected subgroup of $\Aut(X)$. Then $G$ 
	is the countable union of an increasing 
	filtration by  commutative connected algebraic subgroups of $\Aut(X)$.
\end{corollary}

\begin{proof}
	By~\cite[Remark~2.2.3 and Proposition~1.6.3]{FuKr2018On-the-geometry-of} 
	there exists an increasing filtration
	by irreducible subvarieties $V_1 \subseteq V_2 \subseteq \cdots \subseteq G$ such that 
	$G = \bigcup_k V_k$ and $\id_X \in V_k$ for all $k \geq 1$. By Theorem~\ref{thm.CRX},
	the subgroup $G_k \coloneqq \langle V_k \rangle$ of $\Aut(X)$ is algebraic for all $k \geq 1$.
	By \cite[\S7.5 Proposition]{Hu1975Linear-algebraic-g}, $G_k$ is connected for all $k \geq 1$.
	Now, $G_1 \subseteq G_2 \subseteq \cdots \subseteq G$ is our desired filtration.
\end{proof}

Let $H \subseteq \Aut(X)$ be a subgroup, where $X$ is an irreducible affine variety.
A non-trivial commutative unipotent algebraic subgroup $V$ in $\Aut(X)$ is a 
\emph{generalized $H$-root subgroup of weight $\lambda \colon H \to \GG_m$} if for some (and hence every)
isomorphism of algebraic groups $\varepsilon \colon \GG_a^{\dim V} \to V$ we have
\[
h \circ \varepsilon(s) \circ h^{-1} = \varepsilon(\lambda(h) \cdot s) \quad
\textrm{for all $s \in \GG_a^{\dim V}$ and all $h \in H$} \, .
\]	
In case $H$ is an algebraic subgroup of $\Aut(X)$
the abstract group homomorphism $\lambda \colon H \to \GG_m$ is a homomorphism of algebraic groups.
In case $\lambda \colon H \to \GG_m$ is the trivial weight, we say that
$V$ is a \emph{trivial generalized $H$-root subgroup}.
An \emph{$H$-root subgroup} is just a generalized  $H$-root subgroup $V$ where $\dim V = 1$. 

\section{Notation and setup concerning affine toric varieties}
\label{sec.Notations}

Throughout this section, $T$ denotes an algebraic torus of dimension
$n \geq 1$ and $X$ denotes an affine $T$-toric variety, i.e. $X$ is a normal affine variety 
endowed with a faithful $T$-action that admits a dense open orbit.
We see $T$ as an algebraic subgroup of $\Aut(X)$ via the fixed $T$-action on $X$.
Moreover, we fix $x_0$ in the open $T$-orbit of $X$ and denote by
$\eta_{x_0} \colon T \to X$, $t \mapsto t x_0$ the orbit map with respect to $x_0$.
Hence, $\eta_{x_0}$ is an open embedding of $T$ into $X$.

We refer the reader to~\cite{Fu1993Introduction-to-to}
for the basic notions of affine toric varieties.

\subsection{General notions}

Denote by $M = \frak{X}(T)$ 
the character group of $T$ (which is additively written) and by $N \coloneqq \Hom_{\ZZ}(M, \ZZ)$ the group
of one-parameter subgroups of $T$. Moreover, let $M_{\RR} \coloneqq M \otimes_{\ZZ} \RR$, 
$N_{\RR} \coloneqq N \otimes_{\ZZ} \RR = \Hom_{\ZZ}(N, \RR)$ and denote by
\[
M_{\RR} \times N_{\RR} \to \RR \, , \quad (u, v) \mapsto \sprod{u}{v} \coloneqq v(u)
\]
the natural perfect pairing between $M_{\RR}$ and $N_{\RR}$. 
There exists a unique
strongly convex rational polyhedral cone $\sigma \subseteq N_{\RR}$ such that 
its dual $\sigma^\vee \subseteq M_{\RR}$ satisfies 
\[
\kk[X] = \kk[\sigma^\vee \cap M] \subseteq \kk[M] = \kk[T] \, ,
\]
where $\kk[X]$ denotes the coordinate ring of $X$,
$\kk[M]$ denotes the group algebra of $M$ over $\kk$ and the
inclusion $\kk[X] \subseteq \kk[M]$ is given by the comorphism
of the open embedding $\eta_{x_0} \colon T \to X$. Hence, 
$\kk[X]$ is the subalgebra of those functions in $\kk[M]$ that extend to $X$
via the open embedding $\eta_{x_0} \colon T \to X$. 

If $\rho$ is an extremal ray of $\sigma$, then $\rho^\perp \subseteq M_{\RR}$ denotes the subset of 
all $u \in M_{\RR}$ such that $\langle u, v \rangle = 0$ for all $v \in \rho$.

	\subsection{Weights of $T$-root subgroups}
\label{subsec.Notation_Affine_toric_root_subgroups}
%
For an extremal ray $\rho \subseteq \sigma$ 
define
\[
S_\rho \coloneqq \set{m \in (\sigma_\rho)^\vee \cap M}{\sprod{m}{v_\rho} = -1} \subseteq M \, ,
\]
where $\sigma_\rho$ denotes the rational convex polyhedral cone in $N_{\RR}$ spanned by
all the extremal rays of $\sigma$ except $\rho$ and $v_\rho$ denotes the unique primitive
vector inside $\rho \cap N$. 
If $\rho_1, \ldots, \rho_r$ denote the extremal rays 
of $\sigma$, then $S_{\rho_1}, \ldots, S_{\rho_r}$ are pairwise disjoint, the union is equal to 
the set of weights of $T$-root subgroups 
and for every such weight $e$ there exists only one $T$-root subgroup
of weight $e$, see \cite[Lemma~2.6, Theorem~2.7]{Li2010Affine-Bbb-T-varie}.
%

\subsection{Basic results on affine toric varieties}

In this section we gather some basic results on affine toric varieties, which will be
used later in this paper.

\begin{lemma}
	\label{lem.span_facet}
	Let $\rho$ be an extremal ray of $\sigma$. Then the group generated by
	$\sigma^\vee \cap \rho^\bot \cap M$ is equal to $\rho^\bot \cap M$.
\end{lemma}

\begin{proof}
	Let $S \subseteq \rho^\bot \cap M$ be the subgroup generated by $\sigma^\vee \cap \rho^\bot \cap M$.
	Let $m \in \rho^\bot \cap M$. Take 
	$m_1 \in M$ in the relative interior of $\sigma^\vee \cap \rho^\bot$. 
	Then, there exists $t \in \ZZ_{\geq 0}$
	such that $\sprod{m+tm_1}{v_\mu} \geq 0$ for all extremal rays $\mu$ of $\sigma$. 
	Hence, $m+tm_1 \in \sigma^\vee \cap \rho^\bot \cap M \subseteq S$.
	On the other hand $tm_1 \in S$ and thus we get $m \in S$.
\end{proof}

\begin{lemma}
 \label{lem.Simultaneous_kernels}
Let $e_1, \ldots, e_k$ in $M = \frak{X}(T)$, where $k \geq 1$. Then 
\[
K \coloneqq  \bigcap_{i=1}^k \ker(e_i) \simeq \GG_m^{n-r} \times F
\quad \textrm{and} \quad
\frak{X}(K) \simeq M / \Span_{\ZZ} \{e_1, \ldots, e_k \}
\]
for some finite commutative group $F$, where $r = \rank \Span_{\ZZ} \{e_1, \ldots, e_k \}$.
\end{lemma}

\begin{proof}
Consider the homomorphism
$e \coloneqq (e_1, \ldots, e_k) \colon T \to \GG_m^k$ of algebraic groups.
The exact sequence $1 \to K \to T \stackrel{e}{\to} \GG_m^k$
induces an exact sequence between character groups
$\frak{X}(\GG_m^k) \stackrel{e^\ast}{\to} M \to \frak{X}(K) \to 0$, as
$\mathfrak{X}(D)$ forms a basis of $\kk[D]$ for all diagonalizable algebraic groups $D$.
Writing $\mathfrak{X}(T) = \ZZ^n$ and $\mathfrak{X}(\GG_m^k) = \ZZ^k$, the homomorphism
$e^\ast$ identifies with
$\ZZ^k \to \ZZ^n$, $(a_1, \ldots, a_k) \mapsto \sum_{i = 1}^k a_i e_i$.
Hence, 
$\mathfrak{X}(K) \simeq M / e^\ast(\mathfrak{X}(\GG_m^k)) = \ZZ^n / \Span_{\ZZ} \{e_1, \ldots, e_k\}$
has rank $n-r$ and thus $\dim K = n-r$.
This implies the statement.
\end{proof}	

\begin{lemma}
\label{lem.Existence_lin_indep_characters}
Let $U$ be a $T$-root subgroup in $\Aut(X)$. 
Then there exist $T$-root subgroups $U_1 \coloneqq U, U_2, \ldots, U_n \subseteq \Aut(X)$
that commute pairwise 
such that the corresponding
weights in $M$ are linearly independent.
\end{lemma}

\begin{proof}
Let $e_0 \in S_{\rho}$ be the weight corresponding to $U$. Choose 
$a_2, \ldots, a_{n} \in \sigma^\vee \cap \rho^\bot \cap M$ that are linearly independent;
see Lemma~\ref{lem.span_facet}. Then the elements
\[
e_0 \, , \ e_0 + a_2 \, , \ e_0 + a_3 \, , \ \ldots \, , \ e_0 + a_n \in S_\rho \subseteq M
\]
are linearly independent in $M$. Hence, we may choose $U_i$
as the root subgroup corresponding to $e_0 + a_i \in S_\rho$ for $2 \leq i \leq n$.
\end{proof}

\begin{lemma}
	\label{Lem.Characterize_Torus}
	The affine toric variety $X$ is non-isomorphic to $T$ if and only if $\Aut(X)$
	contains a $T$-root subgroup.
\end{lemma}

\begin{proof}
	If $X$ is not isomorphic to $T$, then $\Aut(X)$ contains a $T$-root subgroup
	by \cite[Remark 2.5]{Li2010Affine-Bbb-T-varie}. The reverse direction is clear.
\end{proof}

\begin{lemma}
	\label{Lem.Dense_open_DU-orbit}
	Let $U$ be a $T$-root subgroup of $\Aut(X)$ and let $D \subseteq T$ be the closed subgroup
	of those elements in $T$ that commute with $U$. Then  $D$
	is connected and $DU$ acts with a dense open orbit
	on $X$.
\end{lemma}

\begin{proof}
	Let $\rho$ be the extremal ray of $\sigma$ such that the $T$-weight $e$ of $U$
	lies inside $S_\rho$. Since $e$ is a primitive vector in $M$ and since
	$D$ is the kernel of $e$, it follows that $D$ is connected.

	By~\cite[Lemma 2.6]{Li2010Affine-Bbb-T-varie} we have that the $\kk$-algebra
	$\kk[X]^U$ of $U$-invariants is equal to $\kk[\sigma^\vee \cap \rho^\bot \cap M]$. 
	As $\sigma^\vee \cap \rho^\bot \cap M$ is a finitely generated semi-group by 
	Gordon's Lemma (see e.g.~\cite[\S1.2, proof of Proposition~1]{Fu1993Introduction-to-to}), it follows
	that $\kk[X]^U$ is a finitely generated $\kk$-algebra.
	
	Let $t \in D$. Then $e(t) = 1$.
	If $t$ acts trivially on $\kk[X]^U$, then we have $u(t) = 1$ for all 
	$u \in \sigma^\vee \cap \rho^\bot \cap M$. By Lemma~\ref{lem.span_facet} we get thus
	$u(t) =1$ for all $u \in \rho^\bot \cap M$. Since $\Span_{\ZZ}( \{e\} \cup (\rho^\bot \cap M)) = M$,
	it follows that $t$ is the identity element in $D$. Hence, $D$ acts faithfully on $\kk[X]^U$
	and thus also on $X \aquot U = \Spec(\kk[X]^U)$.
	
	Since the algebraic quotient $X \to X \aquot U$ restricts to a trivial principal $U$-bundle
	over some open dense subset of $X \aquot U$ (see~e.g.~\cite[\S1.4, Principle 11]{Fr2017Algebraic-theory-o})
	and since $D$ acts with a dense open orbit
	on $X \aquot U$, we get the second statement.
\end{proof}

\section{Results on centralizers in an automorphism group}
\label{Sect.Centralizers}
In this section we study centralizers of certain closed subgroups in the automorphism group of an irreducible affine variety
with a particular emphasis on centralizers of commutative unipotent subgroups. 
The definitions of $\Repl_X(U)$, $\rRepl_X(U)$ and $\A_X(U)$ were given in the introduction.

\begin{proposition}\label{Prop.centralizer}
	Let $X$ be an irreducible affine variety and let $U \subseteq \Aut(X)$ be a commutative unipotent algebraic subgroup.
	Then $\Cent_{\Aut(X)}(\Repl_X(U)) = \A_X(U)$ and $\A_X(U)$ is commutative.
	In particular, $\A_X(U)$ is maximal with respect to be a commutative subgroup of $\Aut(X)$
	and $\Cent_{\Aut(X)}(\A_X(U)) = \A_X(U)$.
\end{proposition}

As $\A_X(U)$ is a commutative subgroup of $\Aut(X)$ that contains $\Repl_X(U)$ we have
\[
\A_X(U) \subseteq \Cent_{\Aut(X)}(\Repl_X(U)) \, .
\]
Thus we only have to prove the reverse inclusion. For the proof we need the following result
about the generic structure of the algebraic quotient by a unipotent algebraic group action on an irreducible 
affine variety:

\begin{proposition}[{\cite[Proposition~1.6]{GrPf1993Geometric-Quotient}}]		
	\label{Prop.Geometric_quotient}
	Let $U$ be a unipotent algebraic group that acts on an irreducible affine variety $X$ and let 
	$\pi \colon X \to X \aquot U \coloneqq \Spec(\kk[X]^{U})$
	be the algebraic quotient. Then there exists a dense open $U$-invariant subset $X^s \subseteq X$ such that
	$\pi(X^s)$ is an open dense subset in $X \aquot U$, $\pi |_{X^s} \colon X^s \to \pi(X^s)$
	is a geometric quotient for the $U$-action on $X^s$ and the following universal property is satisfied:
	For all open subsets $V \subseteq X \aquot U$ such that
	the restriction 
	\[
	\pi |_{\pi^{-1}(V)} \colon \pi^{-1}(V) \to V
	\]
	is a geometric quotient for the $U$-action on $\pi^{-1}(V)$ we have 
	$\pi^{-1}(V) \subseteq X^s$.
\end{proposition}	

\begin{remark}
	\label{Rem.Geometric_quotient}
	By \cite[Remarks~1.7(1)]{GrPf1993Geometric-Quotient} 
	it follows that there exists an open dense subset $V \subseteq X \aquot U$ such that
	$\pi |_{\pi^{-1}(V)} \colon \pi^{-1}(V) \to V$ is a geometric quotient for the $U$-action
	on $\pi^{-1}(V)$.
\end{remark}

\begin{proof}[Proof of Proposition~\ref{Prop.centralizer}]
	As $U$ is a commutative unipotent algebraic group we may identify it with a  $\kk$-vector space where the 
	group-operation is given by addition and the closed subgroups of $U$ are the $\kk$-subspaces.
	
	Let $\varphi \in \Cent_{\Aut(X)}(\Repl_X(U))$. We have to shows that $\varphi \in \A_X(U)$.
	
	\begin{claim}
		\label{Claim.varphi_fixes_fibres}
		For all $g \in \kk[X]^U$ we have $\varphi^\ast(g) = g$. In particular, $\varphi$ fixes the fibres of
		the algebraic quotient $\pi \colon X \to X \aquot U$.
	\end{claim}
	
	Indeed, let $D_u$ be the locally nilpotent derivation of $\kk[X]$ associated to some $u \in U \setminus \{ \id_X\}$.
	Hence $D_u$ is non-zero and thus we may choose $h \in \kk[X]$ with $D_u(h) \neq 0$. This gives 
	for all $g \in \kk[X]^U$
	\begin{eqnarray*}
		g \varphi^\ast(D_u(h)) = (gD_u)(\varphi^\ast(h)) = (\varphi^\ast \circ gD_u)(h) = \varphi^\ast(g) \varphi^\ast(D_u(h)) \, ,
	\end{eqnarray*}
	since $\varphi$ commutes with $u$ and $g \cdot u$. Since $\varphi^\ast(D_u(h)) \neq 0$, the claim follows.
	
	\medskip
	
	By Proposition~\ref{Prop.Geometric_quotient} and Remark~\ref{Rem.Geometric_quotient}, 
	there exists an open dense subset
	$B \subseteq X \aquot U$, such that the restriction
	$\pi_0 \coloneqq \pi |_{\pi^{-1}(B)} \colon \pi^{-1}(B) \to B$ is a geometric quotient and there exists an integer $n \geq 0$ and 
	an isomorphism $\rho \colon \AA^n \times B \to \pi^{-1}(B)$ such that the composition 
	$\pi_0 \circ \rho \colon \AA^n \times B \to B$ is 
	the canonical projection to $B$. As $\varphi$ fixes the fibres of $\pi$ (see Claim~\ref{Claim.varphi_fixes_fibres}), it follows that 
	$\varphi(\pi^{-1}(B)) = \pi^{-1}(B)$. By shrinking $B$ we may further assume that $B$ is smooth and affine.
	
	Let $b_0 \in B$, $x_0 = \rho(0, b_0) \in \pi^{-1}(B)$ and denote by $U_{x_0}$ the stabilizer of $U$ in $x_0$. 
	As $U$ is a commutative unipotent group, there  exists a closed subgroup $U' \subseteq U$
	such that $U' + U_{x_0} = U$ and $U' \cap U_{x_0} = \{0\}$. 
	In particular (as $\pi_0$ is a geometric quotient), $\dim U' = n$. Consider the following closed subvariety of $U' \times B$
	\[
	E \coloneqq \set{(u', b) \in U' \times B}{ u' \rho(0, b) = \rho(0, b) }
	\]
	and the projection
	\[
	p \colon E \to B \, , \quad (u', b) \mapsto b \, .
	\]
	Then every fibre of $p$ is irreducible and $p^{-1}(b_0) = \{(0, b_0)\}$.
	Let $d \geq 0$ be the dimension of a general fibre of $p$ (note that $B$ is irreducible and thus such a $d$ exists).
	Let $F \subseteq E$ be an irreducible component of $E$ with $\dim F = \dim E$. Then 
	the general fibre of $p |_F \colon F \to B$ has dimension $d$ and this map is dominant. As all fibres of $p$
	are irreducible, there exists an open dense subset $B' \subseteq B$ with $p^{-1}(B') \subseteq F$.
	Since $\{0\} \times B' \subseteq p^{-1}(B')$ and since $F$ is closed in $E$, we get
	$\{0\} \times B \subseteq F$. Hence, $p^{-1}(b_0) = \{(0, b_0)\} \subseteq F$ and thus $p|_F \colon F \to B$ has at least one fibre of dimension $0$.
	By the semi-continuity of the fibre dimension \cite[Theorem~14.112]{GoWe2020Algebraic-geometry}, 
	it follows that $d = 0$ (since $F$ is irreducible). Hence, we may shrink $B$ in such a way that $E = \{0\} \times B$. This implies that
	\[
	\eta \colon U' \times B \to \pi^{-1}(B) \, , \quad (u', b) \mapsto u' \rho(0, b)
	\]
	is a bijective $U'$-equivariant morphism and since $U' \times B$, $\pi^{-1}(B) \simeq \AA^n \times B$ are irreducible and smooth it follows from
	Zariski's theorem \cite[Corollary~12.88]{GoWe2020Algebraic-geometry} that $\eta$ is in fact a $U'$-equivariant isomorphism.
	
	Then $\psi \coloneqq \eta^{-1} \circ \varphi \circ \eta$ is an automorphism of $U' \times B$ 
	that fixes the fibres of the canonical projection $U' \times B \to B$. Since $\varphi$ commutes with the $U'$-action on $\pi^{-1}(B)$, it follows that 
	$\psi$ commutes with the $U'$-action on $U' \times B$. Write $\psi(u', b) = (\psi_1(u', b), b)$
	for all $(u', b) \in U'\times B$, where $\psi_1 \colon U' \times B \to U'$ is a morphism. 
	Choose a $\kk$-bases $(u_1, \ldots, u_n)$ in  $U'$. Then there exist $f_1, \ldots, f_n \in \kk[B] \subseteq \kk(X)^U$
	such that $\psi_1(0, b) = \sum_{i=1}^n f_i(b) u_i$ for all $b \in B$.
	Since $\eta$ is $U'$-equivariant, we have for all $b \in B$
	\begin{eqnarray*}
		\varphi(\eta(0, b)) =  \eta(\psi(0, b)) &=& \eta \left( \sum_{i=1}^n f_i(b) u_i, b \right) = \left(\sum_{i=1}^n f_i(b) u_i\right) \eta(0, b) \\
		&=& (f_1 \cdot u_1) \circ \ldots \circ (f_n \cdot u_n) (\eta(0, b)) \, .
	\end{eqnarray*}
	Since $\eta$, $\varphi$ and $(f_1 \cdot u_1) \circ \ldots \circ (f_n \cdot u_n)$ are $U'$-equivariant
	and since the image of  $\{0\} \times B$ unde the $U'$-action on $U' \times B$ is the whole $U' \times B$, it follows that
	\[
	\varphi = (f_1 \cdot u_1) \circ \ldots \circ (f_n \cdot u_n) \in \langle \rRepl_X(U) \rangle \, . 
	\]
	By assumption $\varphi \in \Aut(X)$ and thus $\varphi \in \Aut(X)  \cap \langle  \rRepl_X(U)  \rangle= \A_X(U)$.
\end{proof}

\begin{remark}
	\label{Rem.maximal}
	From Proposition~\ref{Prop.centralizer} it follows that  $\A_X(U)$ is closed in $\Aut(X)$ and consists only of unipotent elements. In particular, $\A_X(U)$ is a commutative unipotent subgroup of $\Aut(X)$.
\end{remark}

For future use we insert here the following statement related to the centralizer of $\A_X(U)$
in case $U$ is one-dimensional.

\begin{lemma}
	\label{Lem.Representation_of_on-parameter-subgrp}
	Let $X$ be an irreducible affine variety, $H \subseteq \Aut(X)$ a connected algebraic subgroup that
	acts with a dense open orbit on $X$, 
	and $U \subseteq \Aut(X)$ a non-trivial $H$-root subgroup of weight $\chi \colon H \to \GG_m$.
	Then 
	\[
	U = \Cent_{\Aut(X)}(\ker(\chi) \cup \A_X(U)) 
	\cap \set{ \varphi \in \Aut(X) }{ \varphi^n = h_n \circ \varphi \circ h_n^{-1}, n \in \NN} \, ,
	\]
	where for each $n \in \NN$ we have chosen an element $h_n \in H$ with $\chi(h_n) = n$.
\end{lemma}

\begin{proof}
	Obviously, the set on the left hand side is included in the set on the right hand side, so let $\varphi$ be an element of the 
	set of the right hand side.
	By Proposition~\ref{Prop.centralizer} there exists $r \in Q(\kk[X]^U)$ with $\varphi = r \cdot u$,
	where $u \in U \setminus \{ \id_X \}$ is a fixed element (because $U$ is one-dimensional). 
	Since $\varphi$ commutes with $\ker(\chi)$, we get 
	that $r$ is a $\ker(\chi)$-invariant. Moreover, since for all $n \in \NN$ we have
	\[
	n r \cdot u = \varphi^n = h_n \circ \varphi \circ h_n^{-1} = n(h_n^\ast)^{-1}(r) \cdot u
	\]
	we get that $(h_n^\ast)^{-1}(r) = r$ for all $n \in \NN$. Since
	the subgroup of $H$ generated by $h_1, h_2, \ldots$ and $\ker(\chi)$ is dense in $H$, we get that
	$r$ is an $H$-invariant. Using that $H$ acts with a dense orbit on $X$ gives us that $r$ is a constant.
\end{proof}

We will need in the future the following statement about centralizers which is 
a generalization of \cite[Lemma~2.10]{KrReSa2021Is-the-affine-spac}.

\begin{proposition}
	\label{Lem.Centralizer_open_dense_orbit}
	Let $X$ be an irreducible affine variety and let  $H \subseteq \Aut(X)$ be an algebraic
	subgroup such that there exists an open dense $H$-orbit $W$ in $X$. 
	
	\begin{enumerate}[wide=0pt, leftmargin=*]	
		\item \label{Lem.Centralizer_open_dense_orbit0}
		Let $x_0 \in W$ and let $S$ be the stabilizer of $x_0$ in $H$. Then
		\begin{equation}
			\label{Eq.action}
			\tag{$\ast$}	
			(S \backslash N_H(S)) \times W \to W \, , \quad (n, hx_0) \mapsto h n^{-1} x_0
		\end{equation}
		is a well-defined faithful algebraic action on $W$, where $N_H(S)$ is the normalizer of $S$ in $H$
		and $S \backslash N_H(S)$ denotes the quotient group of right $S$-cosets.
		The image of the natural induced injective homomorphism $S \backslash N_H(S) \to \Aut(W)$ is equal to $\Cent_{\Aut(W)}(H)$.
		\item \label{Lem.Centralizer_open_dense_orbit0.5} If $W$ is affine, the homomorphism 
		$\varepsilon \colon \Cent_{\Aut(X)}(H) \to \Cent_{\Aut(W)}(H)$, $\varphi \mapsto \varphi |_W$ 
		is a closed immersion of algebraic groups.
		\item \label{Lem.Centralizer_open_dense_orbit1} 
		If $H$ is commutative, then $\Cent_{\Aut(X)}(H)$ is equal to $H$.
		\item \label{Lem.Centralizer_open_dense_orbit2} 
		If $H$ is unipotent, then $\Cent_{\Aut(X)}(H)$ is unipotent as well.
	\end{enumerate}
\end{proposition}

\begin{proof}
	In case $H$ is solvable, $W$ is a affine by \cite[Theorem~1]{Br2021Homogeneous-variet}. Hence in~\eqref{Lem.Centralizer_open_dense_orbit0.5},
	\eqref{Lem.Centralizer_open_dense_orbit1}, \eqref{Lem.Centralizer_open_dense_orbit2} $W$ is affine
	and $\Cent_{\Aut(W)}(H)$ is a closed subgroup of the ind-group $\Aut(W)$.
	
	\eqref{Lem.Centralizer_open_dense_orbit0}: 
	Note that there is a short exact sequence of groups
	\[
	1 \to S \to N_H(S) \xrightarrow{n \mapsto (hx_0 \mapsto hn^{-1}x_0)} \Cent_{\Aut(W)}(H) \to 1 \, .
	\]
	As the morphism $(S \backslash N_H(S)) \times H \to W$, $(n, h) \mapsto hn^{-1}x_0$ is invariant under right-multiplication
	by elements from $S$ on $H$ and as $(S \backslash N_H(S)) \times H \to (S \backslash N_H(S)) \times W$, 
	$(n, h) \mapsto (n, h x_0)$
	is the geometric quotient under this action, it follows that~\eqref{Eq.action} is a morphism.
	Hence,~\eqref{Lem.Centralizer_open_dense_orbit0} follows.
	
	\eqref{Lem.Centralizer_open_dense_orbit0.5}:
	Pre-composition with the open immersion 
	$W \subseteq X$ yields a closed immersion of ind-varieties $\End(X) \to \Mor(W, X)$, 
	see e.g~\cite[Proposition~1.10.1(2)]{FuKr2018On-the-geometry-of}. It restricts to a locally closed immersion
	of ind-varieties
	\[
	\eta \colon \Cent_{\Aut(X)}(H) \to \Mor(W, X) \, ,
	\]
	since $\Cent_{\Aut(X)}(H)$ is closed in $\Aut(X)$ and
	$\Aut(X)$ is locally closed in $\End(X)$, see e.g.~\cite[Theorem~5.2.1]{FuKr2018On-the-geometry-of}.
	Similarly, it follows that $\Cent_{\Aut(W)}(H)$ is locally closed in $\End(W)$. As
	post-composition with $W \subseteq X$ yields a morphism of ind-varieties $\End(W) \to \Mor(W, X)$
	(see \cite[Proposition~1.10.1(1)]{FuKr2018On-the-geometry-of}),
	the restriction to $\Cent_{\Aut(W)}(H)$
	\[
	\rho \colon \Cent_{\Aut(W)}(H) \to \Mor(W, X)
	\]
	is a morphism of ind-varieties as well. 
	Note that $\Cent_{\Aut(W)}(H)$ acts by pre-composition on 
	$\Mor(W, X)$ (again by \cite[Proposition~1.10.1(2)]{FuKr2018On-the-geometry-of}). 
	Then $\rho$ is $\Cent_{\Aut(W)}(H)$-equivariant with respect to this action
	and left-multiplication on $\Cent_{\Aut(W)}(H)$. As $\Cent_{\Aut(W)}(H)$ is an algebraic group (by~\eqref{Lem.Centralizer_open_dense_orbit0}),
	this yields that the image of $\rho$ is open in its closure in $\Mor(W, X)$ and that this image is smooth.
	Thus $\rho$ yields an isomorphism onto its locally closed image in $\Mor(W, X)$. Since
	$\rho \circ \varepsilon = \eta$ and since $\eta$ is a locally closed immersion, 
	it follows that $\varepsilon$ is a locally closed immersion of ind-groups. As the image of
	$\varepsilon$ is a locally closed subgroup in the algebraic group $\Cent_{\Aut(W)}(H)$,
	it follows that this image is closed in $\Cent_{\Aut(W)}(H)$
	(see~e.g.~\cite[Proposition~7.4A]{Hu1975Linear-algebraic-g}) 
	and hence~\eqref{Lem.Centralizer_open_dense_orbit0.5} follows.
	
	\eqref{Lem.Centralizer_open_dense_orbit1}:
	As $H$ is commutative, $S$ is equal to the stabilizer
	of any point of $W$. Since $W$ is dense in $X$, we get thus that $S$ is trivial.
	Hence $\Cent_{\Aut(W)}(H)$ is equal to $H$ inside $\Aut(W)$ by~\eqref{Lem.Centralizer_open_dense_orbit0}.
	Again, since $H$ is commutative, we have 
	\[
	H \subseteq \varepsilon(\Cent_{\Aut(X)}(H)) \subseteq \Cent_{\Aut(W)}(H) =
	H \subseteq \Aut(W) \, .
	\]
	\eqref{Lem.Centralizer_open_dense_orbit2}: This
	follows from~\eqref{Lem.Centralizer_open_dense_orbit0}, \eqref{Lem.Centralizer_open_dense_orbit0.5} 
	and the fact that quotients of unipotent algebraic groups are again unipotent.
\end{proof}

\begin{corollary}
	\label{Cor.Knop}
	Let $X$ be a $G$-spherical affine variety for some connected reductive algebraic subgroup $G \subseteq \Aut(X)$.
	Then $\Cent_{\Aut(X)}(G)$ is a diagonalizable algebraic subgroup of $\Aut(X)$.
\end{corollary}

\begin{proof}
	Let $B \subseteq G$ be a Borel subgroup and let $x_0 \in X$ be a point in the open $B$-orbit of $X$.
	By Proposition~\ref{Lem.Centralizer_open_dense_orbit}\eqref{Lem.Centralizer_open_dense_orbit0.5} it follows that
	the restriction to $B x_0$ yields a homomorphism of algebraic groups 
	$\varepsilon \colon \Cent_{\Aut(X)}(B) \to \Cent_{\Aut(B x_0)}(B)$.
	Denote by $G_{x_0}$ the stabilizer of $x_0$ in $G$. Note that $G_{x_0} \backslash N_G(G_{x_0})$
	acts faithfully on $G x_0$ via~\eqref{Eq.action} (see 
	Proposition~\ref{Lem.Centralizer_open_dense_orbit}\eqref{Lem.Centralizer_open_dense_orbit0}). This action leaves
	the open $B$-orbit $B x_0$ invariant. The natural homomorphism 
	$G_{x_0} \backslash N_G(G_{x_0}) \to \Aut(B x_0)$  induced by this action is injective
	and the image $H$ is the algebraic subgroup in $\Aut(B x_0)$ 
	of those automorphisms of $B x_0$ that
	extend to a $G$-equivariant automorphism of $G x_0$. Using~$\varepsilon \colon \Cent_{\Aut(X)}(B) \to \Cent_{\Aut(B x_0)}(B)$, 
	it follows that
	$\Cent_{\Aut(X)}(G)$ is isomorphic to a closed subgroup of $H$.
	By~\cite[Corollaire~5.2]{BrPa1987Valuations-des-esp} (see also \cite[Theorem~6.1]{Kn1991The-Luna-Vust-theo}) 
	the group $N_G(G_{x_0}) / G_{x_0}$ is diagonalizable (as  $\kk$ has characteristic zero). The statement follows thus from the 
	fact that $N_G(G_{x_0}) / G_{x_0}$ and $G_{x_0} \backslash N_G(G_{x_0}) \simeq H $ are anti-isomorphic as 
	algebraic groups.
\end{proof}

\begin{lemma}
	\label{Lem.No_B-root subgroups}
	Let $X$ be an affine $G$-spherical variety for some connected reductive algebraic group $G$
	and let $B \subseteq G$ be a Borel subgroup. Then there is no trivial $B$-root subgroup.
\end{lemma}

\begin{proof}
	Assume $U \subseteq \Aut(X)$ is a $B$-root subgroup of trivial $B$-weight.
	By \cite[Proposition~5.1(b)]{ArAv2022Root-subgroups-on-} the $B$-root subgroup
	$U$ is in fact a $G$-root subgroup. 
	By \cite[Proposition~6.7]{ArAv2022Root-subgroups-on-} $U$ is horizontal as
	a $B$-root subgroup, i.e.~the locally nilpotent derivation on $\kk[X]^{R_u(B)}$ associated to $U$ 
	is non-trivial, where $R_u(B)$ is the unipotent radical of $B$. 
	However, this implies that the $B$-weight of $U$ is non-trivial 
	(e.g.~by Proposition~\ref{Lem.Centralizer_open_dense_orbit}\eqref{Lem.Centralizer_open_dense_orbit1}
	applied to the toric variety $\Spec(\kk[X]^{R_u(B)})$) 
	and thus we arrive at a contradiction.
\end{proof}

\section{Maximal commutative unipotent subgroups}
\label{Sect.MCUS}

In this section we describe all commutative unipotent subgroups in the automorphism group of an irreducible affine variety 
that are maximal among those subgroups (see Corollary~\ref{Cor.Maximal2}) and we prove that these groups coincide with their centralizers. These results are based on our study of centralizers of commutative 
unipotent subgroups
from Sect.~\ref{Sect.Centralizers}.
As a consequence, we show for irreducible affine varieties $X, Y$ where the ground field $\kk$
is uncountable the following: under a group isomorphism
$\Aut(X) \to \Aut(Y)$, unipotent elements are mapped onto unipotent elements (see Corollary~\ref{Cor.UnipotentToUnipotent}). 

\begin{lemma}
	\label{Cor.Maximal}
	Assume $X$ is an irreducible affine variety and let $G$ be a closed commutative unipotent subgroup of $\Aut(X)$. Then there exists
	a (commutative unipotent) algebraic subgroup $U \subseteq G$ such that $G \subseteq \A_X(U)$ and
	$\kk[X]^U = \kk[X]^G$. 
	
\end{lemma}

\begin{proof}
	Write $G  = \bigcup_{i \geq 0} G_i$, where $G_i$ is a commutative unipotent algebraic subgroup of $\Aut(X)$,
	see Corollary~\ref{cor.CRX}.
	Take successive algebraic subgroups $U_1 \subseteq U_2 \subseteq \cdots$ inside $G$ such that
	$\dim U_i = i$ and $\kk[X]^{U_i} \subsetneq  \kk[X]^{U_{i-1}}$ for all $i \geq $1. By \cite[Priniciple~11(e)]{Fr2017Algebraic-theory-o}
	we have that the Krull dimension of $\kk[X]^{U_{i-1}}$ is one more than that of $\kk[X]^{U_{i}}$. Hence the sequence
	$U_1 \subseteq U_2 \subseteq \cdots$ stops with some $U \coloneqq U_k$ and $\kk[X]^U = \kk[X]^G$.
	Therefore $\kk(X)^U = Q(\kk[X]^U) = Q(\kk[X]^G)$ and every element $r \cdot u$ with $u \in U$ and $r \in \kk(X)^U$ commutes with $G$.
	This shows that $\A_X(U)$ commutes with $G$. Because $\Cent_{\Aut(X)}(\A_X(U)) = \A_X(U)$ (by Proposition~\ref{Prop.centralizer}) it follows that
	$G \subseteq \A_X(U)$.
\end{proof}

\begin{corollary}
	\label{Cor.Maximal2}
	Assume $X$ is an irreducible affine variety and let $G$ be a commutative unipotent subgroup
	in $\Aut(X)$. Then the following statements are equivalent:
	\begin{enumerate}[label=\alph*), wide=0pt, leftmargin=*]
		\item \label{Cor.Maximal2a} $G$  is maximal among commutative subgroups in $\Aut(X)$;
		\item \label{Cor.Maximal2b} $G$ is maximal among commutative unipotent subgroups in $\Aut(X)$;
		\item \label{Cor.Maximal2c} $G = \A_X(U)$ for some  commutative unipotent algebraic subgroup $U \subseteq \Aut(X)$.
	\end{enumerate}
\end{corollary}

\begin{proof}
	\ref{Cor.Maximal2a} $\Rightarrow$ 	\ref{Cor.Maximal2b} is obvious, \ref{Cor.Maximal2b} $\Rightarrow$ 	\ref{Cor.Maximal2c} follows from
	Lemma~\ref{Cor.Maximal} and Remark~\ref{Rem.maximal}, and \ref{Cor.Maximal2c} $\Rightarrow$ 	\ref{Cor.Maximal2a} follows from 
	Proposition~\ref{Prop.centralizer}.
\end{proof}

\begin{corollary}
	\label{Cor.UnipotentToUnipotent}
	Assume that $\kk$ is uncountable. Let
	$X, Y$ be irreducible affine varieties and let $\theta \colon \Aut(X) \to \Aut(Y)$ be a group isomorphism. 
	Then $\theta$ maps unipotent elements to unipotent elements.
\end{corollary}

\begin{proof}
	Assume $u \in \Aut(X)$ is a unipotent element, i.e.~$U \coloneqq \overline{\langle u \rangle} \simeq \mathbb{G}_a$,
	where the closure is taken inside $\Aut(X)$. 
	Since $\A_X(U)$ is equal to its own centralizer in $\Aut(X)$ (see Proposition~\ref{Prop.centralizer}), 
	the image $\theta(\A_X(U))$ is a closed commutative subgroup of $\Aut(Y)$.
	As $\A_X(U)$ contains no non-trivial element of finite order, the connected component
	$\theta(\A_X(U))^\circ$ is a commutative unipotent subgroup of $\Aut(Y)$.
	
	\begin{claim}
		\label{Claim.maximal}
		The subgroup $\theta(\A_X(U))^\circ$ of $\Aut(Y)$ 
		is maximal among commutative connected subgroups of $\Aut(Y)$.
	\end{claim}
	
	Indeed, let $H$ be a commutative connected subgroup
	of $\Aut(Y)$ that contains  $\theta(\A_X(U))^\circ$.  Then $\theta^{-1}(H)$ is a commutative subgroup
	of $\Aut(X)$ that contains $S \coloneqq \theta^{-1}(\theta(\A_X(U))^\circ)$. As $S$
	has countable index in $\A_X(U)$, since $\A_X(U)$ is connected, and since $\kk$ is uncountable,
	it follows that  $S$ is dense in $\A_X(U)$, see~e.g.~\cite[Corollary~2.2.2]{FuKr2018On-the-geometry-of}. 
	Hence
	\[
	\theta^{-1}(H) \subseteq \Cent_{\Aut(X)}(S) = \Cent_{\Aut(X)}(\A_X(U)) 
	\xlongequal{\textrm{Prop.~\ref{Prop.centralizer}}} \A_X(U) \, .
	\]
	As $H$ is connected we get thus $H \subseteq \theta(\A_X(U))^\circ$.
	This proves Claim~\ref{Claim.maximal}.
	
	\medskip
	
	Using Claim~\ref{Claim.maximal} and Corollary~\ref{Cor.Maximal2} it follows that $\theta(\A_X(U))^\circ$ is maximal among
	commutative
	subgroups of $\Aut(Y)$ and hence $\theta(\A_X(U))^\circ = \theta(R_X(U))$. This implies that
	$\theta(u)$ is unipotent in $\Aut(Y)$.	
\end{proof}
	
\section{Existence result on root subgroups}

In this section we provide a construction for root-subgroups inside the automorphism group of an irreducible affine variety.

\begin{proposition}
	\label{Prop.Existence_of_root_subgroups}
	Let $X$ be an irreducible affine variety and let
	$H \subseteq \Aut(X)$ be a solvable algebraic subgroup with a non-trivial unipotent radical $R_u(H)$.
	Assume that $H = F H^\circ$ for a finite central subgroup $F \subseteq H$.
	
	\begin{enumerate}[wide=0pt, leftmargin=*]
		\item \label{Prop.Existence_of_root_subgroups1}
		Assume there exists a diagonalizable subgroup $D \subseteq H$ of dimension $r$ such that 
		one of the following conditions holds:
		\begin{enumerate}[wide=0pt, leftmargin=*, label=\alph*)]
			\item $D$ commutes with $R_u(H)$ or
			\item $\dim X = r$ and $R_u(H)$ is one-dimensional.
		\end{enumerate}
		Then there exist $H$-root subgroups $U_1, \ldots, U_r \subseteq \Aut(X)$ of
		linearly independent $D$-weights $\lambda_1, \ldots, \lambda_{r} \colon D \to \GG_m$.
		\item \label{Prop.Existence_of_root_subgroups2}
		Let $T \subseteq H$ be a maximal torus and let $T_1, \ldots, T_n \subseteq T$ be non-trivial subtori.
		Then there exists a $T$-root subgroup in $\Aut(X)$ that 
		doesn't commute with $T_i$ for all $i = 1, \ldots, n$.
	\end{enumerate}
\end{proposition}

\begin{proof}
	The following claim turns out to be very useful for the proof:
	
	\begin{claim}
		\label{Claim.D_acts_faithfully_on_the_quotient}
		Let $V \subseteq R_u(H)$ be a closed subgroup and let 
		$E \subseteq H$ be a diagonalizable subgroup that commutes with $V$. Then 
		$E$ acts faithfully on $X\aquot V$.
	\end{claim}
	
	Indeed, denote by $\pi \colon X \to  X\aquot V$ the algebraic quotient with respect to the $V$-action on $X$,
	let $X^s \subseteq X$ be the open dense subset
	of Proposition~\ref{Prop.Geometric_quotient} and let $\pi' \coloneqq \pi |_{X^s} \colon X^s \to \pi(X^s)$ be the induced geometric quotient. We show now, that the subgroup of elements of finite order of $E$
	acts faithfully on $X \aquot V$. This will imply
	Claim~\ref{Claim.D_acts_faithfully_on_the_quotient} as $E$ is diagonalizable.
	
	So, let $e \in E$ be an element of finite order that acts trivially on 
	$X \aquot V$. By assumption it commutes
	with the $V$-action on $X$ and in particular it leaves $X^s$ and $\pi(X^s)$ invariant.
	As $e$ acts trivially on $X \aquot V$, it leaves every fibre of $\pi'$ invariant. We apply Proposition~\ref{Lem.Centralizer_open_dense_orbit}\eqref{Lem.Centralizer_open_dense_orbit2} 
	to each fibre of $\pi'$ in order to conclude that
	$e$ acts trivially on each fibre of $\pi'$. Since $X^s$ is dense in $X$, it follows that $e$ acts trivially on $X$ and hence $e$ is trivial (as $E$ acts faithfully on $X$).
	
	\medskip
	
	We fix an $H^\circ$-root subgroup $U$ in the center of $R_u(H)$. 
	Since $H = F H^\circ$ for a finite central subgroup $F$ in $H$, it follows that $U$ is 
	an $H$-root subgroup of $R_u(H)$.
	Now, we prove~\eqref{Prop.Existence_of_root_subgroups1} and
	\eqref{Prop.Existence_of_root_subgroups2} separately:
	
	\medskip
	
	\eqref{Prop.Existence_of_root_subgroups1}:
	If $\dim X = r$ and $R_u(H)$ is one-dimensional, then $X$ is $D^\circ$-toric and thus~\eqref{Prop.Existence_of_root_subgroups1} 
	holds by Lemma \ref{lem.Existence_lin_indep_characters} applied to the $D^\circ$-root subgroup $U = R_u(H)$
	(note that in this case $H = H^\circ = D^\circ U$ by Proposition~\ref{Lem.Centralizer_open_dense_orbit}\eqref{Lem.Centralizer_open_dense_orbit1} applied to $D^\circ$, 
	as $F$ is central in $H$).
	Hence, we may assume that
	$D$ commutes with $R_u(H)$.
	We apply Claim~\ref{Claim.D_acts_faithfully_on_the_quotient} to $V = R_u(H)$ and 
	$E = D$ in order to get 
	\begin{equation}
		\label{Eq.DImension_bound}
		\dim X \aquot R_u(H) - \dim (X \aquot R_u(H))\aquot D  \geq \dim D = r \, .
	\end{equation}
	Let $L \subseteq H$ be a diagonalizable subgroup of $H$ such that $H = L R_u(H)$
	(in fact, one can choose a maximal subtorus $T$ of $H^\circ$ and take $L = F T$).
	Then $\kk[X]^{R_u(H)}$ is generated as a $\kk$-vector space by 
	$L$-semi-invariants of $\kk[X]^{R_u(H)}$. 
	As every $L$-semi-invariant of $\kk[X]^{R_u(H)}$ is an $H$-semi-invariant,
	by using~\eqref{Eq.DImension_bound} we get
	$H$-semi-invariant functions $f_1, \ldots, f_r \in \kk[X]^{R_u(H)}$ 
	that are algebraically independent over the quotient field of $\kk[X]^{DR_u(H)}$.
	Denote by $\lambda_1,\ldots, \lambda_r \colon D \to \GG_m$
	the $D$-weights of $f_1, \ldots, f_r$, respectively. If $\sum_{i=0}^r a_i \lambda_i = 0$ for integers
	$a_1, \ldots, a_r$, then $\prod_{i=1}^r f_i^{a_i}$ is a $DR_u(H)$-invariant and hence $a_1 = \ldots = a_r = 0$.
	We proved that $\lambda_1, \ldots, \lambda_r$ are linearly independent in the character group $\frak{X}(D)$.
	For $i = 1, \ldots, r$, let $U_i \coloneqq f_i \cdot U =  \set{f_i \cdot u}{u \in U}$. Hence $U_i$ is an $H$-root subgroup of 
	$D$-weight $\lambda_i$ since $D$ and $U$ commute. 
	This proves~\eqref{Prop.Existence_of_root_subgroups1}.
	
	\medskip
	
	\eqref{Prop.Existence_of_root_subgroups2}:
	After renumbering the $T_i$, we may assume that there exists $m \in \{0, 1, \ldots, n\}$
	such that $T_i$ commutes with $U$ if and only if $1 \leq i \leq m$. If $m = 0$, then~\eqref{Prop.Existence_of_root_subgroups2} is satisfied and thus we may assume $m \geq 1$.
	For each $1 \leq i \leq m$,
	the torus $T_i$ acts faithfully on $X\aquot U$, by Claim~\ref{Claim.D_acts_faithfully_on_the_quotient}
	applied to $V = U$ and $E = T_i$.
	Hence, for $1 \leq i \leq m$ we may choose a $T$-semi-invariant $f_i \in \kk[X]^{U}$, such that $f_i$ is not a $T_i$-invariant
	(otherwise every element in $\kk[X]^{U}$ would be $T_i$-invariant, as $\kk[X]^{U}$ is
	$\kk$-linearly spanned by its $T$-semi-invariants). For $1 \leq i \leq m$ let $\lambda_i \colon T \to \GG_m$
	be the $T$-weight of $f_i$ and let $\lambda \colon T \to \GG_m$ be the $T$-weight of $U$.
	For $i = 1, \ldots, n$, let
	\[
	A_i \coloneqq \Bigset{(a_1, \ldots, a_m) \in \ZZ^m}{\left(\sum_{j=1}^m a_j \lambda_j |_{T_i} \right) + \lambda|_{T_i} = 0}
	\subseteq \ZZ^m \, .
	\]
	Since $\lambda_i |_{T_i} \neq 0$ for all $1 \leq i  \leq m$ and $\lambda |_{T_i} \neq 0$ for all $m < i \leq n$,
	it follows that $A_i$ is either a translate of a rank $(m-1)$-subgroup of $\ZZ^m$ or it  is empty. In particular there
	exists $(a_1, \ldots, a_m) \in \ZZ^m \setminus \bigcup_{i=1}^n A_i$. Consider the $T$-semi-invariant
	\[
	f \coloneqq \prod_{j=1}^m f_j^{a_j} \in \kk[X]^{U}
	\]
	and the $T$-root subgroup $f \cdot U$ of $\Aut(X)$. 
	Since for all $1 \leq i \leq n$, 
	the $T_i$-weight of $f \cdot U$ is equal to $(\sum_{j=1}^m a_j \lambda_j |_{T_i}) + \lambda|_{T_i} \neq 0$,
	it follows that $T_i$ does not commute with $f \cdot U$. 
	This proves~\eqref{Prop.Existence_of_root_subgroups2}.
\end{proof}

\section{Abstract group isomorphisms of algebraic groups}	

In this section, we provide results stating that an abstract group isomorphism of certain algebraic groups
is an isomorphism of algebraic groups up to field automorphisms. In order to do so, we have to introduce the language of base-change via a field automorphism.

\medskip

If $\tau \colon \kk \to \kk$ is a field automorphism
and $X$ is a variety (defined over $\kk$), then $X^{\tau}$ is the unique variety (defined over $\kk$)
such that the following diagram is a pullback diagram
\begin{equation}
	\label{Eq.Def_tau}
	\begin{gathered}
		\xymatrix@=20pt{
			X^\tau \ar[d] \ar[rr]^-{\tau_X} && X \ar[d] \\
			\Spec(\kk) \ar[rr]^-{\Spec(\tau)} && \Spec(\kk) \, ,
		}
	\end{gathered}
\end{equation}
where $\tau_X$ is defined by the pullback diagram (it is not a morphism of $\kk$-varieties!).
Note, if $X' \subseteq X$ is a closed subvariety, then $(X')^\tau = (\tau_X)^{-1}(X')$ is a closed subvariety of $X$. Moreover, if $f \colon X \to Y$ is a morphism of varieties, 
then there exists a unique morphism of varieties $f^\tau \colon X^\tau \to Y^\tau$ such that
$\tau_Y \circ f^\tau = f \circ \tau_X$. The assignment $X \mapsto X^\tau$, $f \mapsto f^\tau$
gives a self-equivalence of the category of varieties.  

\begin{remark}
	\label{Rem.Xtau_X_def_over_Q}
	It may happen, that $X^\tau$ and $X$ are isomorphic  (as varieties over $\kk$) and in this case we simply denote
	$X^\tau$ by $X$. This happens to all varieties $X$ that are defined over $\QQ$, in particular to all toric varieties.
\end{remark}

\begin{remark}
	\label{Rem.Field_Automorphism}
	Let $X$ be an affine variety and let $\tau$ be a field automorphism of $\kk$. 
	For each algebraic subgroup $G \subseteq \Aut(X)$ that corresponds to
	a faithful $G$-action $\rho \colon G \times X \to X$, the map
	\[
	\rho^\tau \colon G^\tau \times X^\tau \xrightarrow{\tau_G \times \tau_X} G \times X \xrightarrow{\rho} X 
	\xrightarrow{(\tau_X)^{-1}} X^\tau
	\]
	is a faithful algebraic $G^\tau$-action  on $X^\tau$. In this way, we see $G^\tau$ as an algebraic subgroup
	of $\Aut(X^\tau)$ and then we have the following commutative 
	diagram of (abstract) groups, where the horizontal arrows are isomorphisms.
	\[
	\xymatrix@R=10pt{
		\Aut(X^\tau) \ar[rrr]^-{\varphi \mapsto \tau_X \circ \varphi \circ (\tau_X)^{-1}}_{\psi^\tau \mapsfrom \psi} 
		&&& \Aut(X) \\
		G^{\tau} \ar[rrr]^-{\tau_G} \ar@{}[u]|-{\bigcup} &&& G \ar@{}[u]|-{\bigcup}
	} 
	\]
	
\end{remark}

\begin{definition}
	Let $X, Y$ be varieties.
	Assume that $\eta \colon X \to Y$ is a map on the sets of $\kk$-points
	and assume there exists a field automorphism $\tau$ of $\kk$ such that the map on the sets of $\kk$-points
	$\eta \circ \tau_X \colon X^\tau \to Y$ is equal to the underlying map on the sets of $\kk$-points 
	of a morphism of varieties $X^\tau \to Y$. 
	If this is the case, we simply say that $\eta \circ \tau_X \colon X^\tau \to Y$ is a \emph{morphism of varieties}.
\end{definition}

\begin{lemma}
	\label{Lem.Same_field_auto}
	Let $H, H'$ be non-finite algebraic groups and let $\eta \colon H \to H'$ be an abstract group isomorphism.
	For $i = 1, 2$, let $G_i$ and $G_i'$ be algebraic groups that contain $H$ and $H'$ as algebraic subgroups, respectively.
	Moreover, let $\eta_i \colon G_i \to G_i'$ be an abstract group isomorphism that restricts to $\eta$ on $H$
	and let $\tau_i$ be a field automorphism of $\kk$
	such that $\eta_i \circ (\tau_i)_{G_i} \colon G_i^{\tau_i} \to G_i'$ is an
	isomorphism of algebraic groups. Then $\tau_1 = \tau_2$.
\end{lemma}

The following diagram illustrates the situation of the lemma, where $i \in \{1, 2\}$:
\[
\xymatrix@R=10pt{
	G_i^{\tau_i} \ar[r]^{(\tau_i)_G} & G_i \ar[r]^{\eta_i} & G_i' \\
	& H \ar@{}[u]|-{\bigcup} \ar[r]^{\eta} & H' \ar@{}[u]|-{\bigcup}
}
\]

\begin{proof}[Proof of Lemma~\ref{Lem.Same_field_auto}]
	Since $H$ is a non-finite algebraic group, there exists a connected algebraic subgroup $K \subseteq H$
	of dimension one. Note that $K^{\tau_i}$ is isomorphic as a variety over $\kk$ to $K$, since $\GG_a$ and $\GG_m$
	are the only connected one-dimensional algebraic groups
	(see \cite[Theorem~20.5]{Hu1975Linear-algebraic-g})
	and they are defined over $\QQ$. 
	Let $C \coloneqq K$ in case $K = \GG_m$ and $C \coloneqq K \setminus \{0\}$ in case $K = \GG_a$. 
	
	Let $\tau \coloneqq \tau_1 \circ (\tau_2)^{-1}$. Then the underlying
	map on the sets of $\kk$-points
	\[
	\tau_C \colon C \xrightarrow{(\tau_1)_C} C \xrightarrow{((\tau_2)_C)^{-1}} C
	\]
	is equal to the underlying map on the sets of $\kk$-points of the isomorphism of varieties (that is defined over $\kk$):
	\[
	f \coloneqq (\eta_2 \circ (\tau_2)_{G_2})^{-1} \circ (\eta_1 \circ (\tau_1)_{G_1}) |_C \colon C \to C \, .
	\]
	Hence, the ring isomorphism
	\[
	\tau_C^\ast \colon \kk[t^{\pm 1}] \to \kk[t^{\pm 1}] \, , \quad \sum_i a_i t^i \mapsto \sum_i \tau(a_i) t^i 
	\]
	induces the same map on maximal ideals as $f^\ast \colon \kk[t^{\pm 1}] \to \kk[t^{\pm 1}]$. 
	Write $f^\ast(t) = a t^{\varepsilon}$ for $a \in \kk^\ast$ and $\varepsilon \in \{\pm 1\}$.
	Now, for all $\lambda \in \kk^\ast$ the maximal ideals $(\tau_C^{\ast}(t - \lambda))$, $(f^\ast(t-\lambda))$ in $\kk[t^{\pm 1}]$
	are the same and hence
	there exist $a_\lambda \in \kk^\ast$, $e_\lambda\in \ZZ$ such that
	\[
	t-\tau(\lambda) = \tau_C^\ast(t-\lambda) = a_{\lambda} t^{e_\lambda} f^\ast(t-\lambda) =
	a a_\lambda t^{e_\lambda + \varepsilon} - a_\lambda \lambda t^{e_\lambda}
	\]
	inside $\kk[t^{\pm 1}]$. 
	\begin{itemize}[wide=0pt, leftmargin=*]
		\item In case $\varepsilon = 1$, it follows that $e_{\lambda} = 0$ and hence $a a_{\lambda} = 1$, $a_{\lambda} \lambda = \tau(\lambda)$ for all $\lambda \in \kk^\ast$.
		This implies that $\tau(\lambda) = a^{-1} \lambda$ for all $\lambda \in \kk^\ast$. As $\tau$ is a field automorphism, we get $a^{-1} = 1$ and hence $\tau_1 = \tau_2$.
		\item In case $\varepsilon = -1$, it follows that $e_{\lambda} = 1$ and hence $-a_{\lambda} \lambda  = 1$, $a a_{\lambda} = -\tau(\lambda)$ for all $\lambda \in \kk^\ast$.
		This implies that $\tau(\lambda) = a \lambda^{-1}$ for all $\lambda \in \kk^\ast$. As $\tau$ is a field automorphism, we arrive at a contradiction. \qedhere
	\end{itemize}
\end{proof}

The following lemma generalizes and uses \cite[Lemma~5.3]{LiReUr2023Characterization-o} in an essential way:

\begin{lemma}
	\label{Lem.Constructing_Iso_of_alg_grps}
	Let $H, H'$ be connected solvable algebraic groups of the same dimension  such that the unipotent radicals $U \subseteq H$,
	$U' \subseteq H'$ are one-dimensional.
	Denote by $T \subseteq H$, $T' \subseteq H'$ maximal tori together with one-dimensional subtori
	$T_1, \ldots, T_n \subseteq T$, $T_1', \ldots, T_n' \subseteq T'$ such that 
	\[
	T = T_1 \cdots T_n \quad  \textrm{and} \quad T' = T_1' \cdots T_n'
	\]
	and $T_i$, $U$ do not commute for all $i$, and $T_i'$, $U'$ do not commute for all $i$.
	If there exists a group isomorphism 
	$\eta \colon H \to H'$ such that $\eta(UT_i) = U'T'_i$
	for all $i=1, \ldots, n$, then there exists a field automorphism $\tau$ of $\kk$ such that  
	$\eta \circ \tau_{H} \colon H^\tau \to H'$
	is an isomorphism of algebraic groups.
\end{lemma}

\begin{proof}
	Let $\eta_i \colon U T_i \to U' T'_i$ be
	the restriction of $\eta$ to $U T_i$ for $i=1, \ldots, n$. Note that
	the center of $UT_i$ is equal to $Z_i \coloneqq \ker(\chi_i)$ and
	the center of $U' T'_i$ is equal to $Z'_i \coloneqq \ker(\chi'_i)$,
	where $\chi_i \colon T_i \to \GG_m$ denotes the $T_i$-weight of $U$
	and $\chi'_i \colon T'_i \to \GG_m$ the $T'_i$-weight of $U'$.
	Hence, $\eta_i$ restricts to an isomorphism $Z_i \simeq Z'_i$ and by assumption, $Z_i$, $Z_i'$ are both finite.
	Hence, there exists an isomorphism of algebraic groups $\iota_i \colon T_i' \to T_i$ such that
	$\chi_i' = \chi_i \circ \iota_i$. Fix an isomorphism of algebraic groups $\varepsilon \colon U' \to U$.
	Then $U' T'_i \to U T_i$, $u't' \mapsto \varepsilon(u') \iota_i(t')$ is an isomorphism of algebraic groups
	(where $u' \in U'$ and $t' \in T'_i$). 
	By \cite[Lemma~5.3]{LiReUr2023Characterization-o} there
	exists a field automorphism $\tau_i$ of $\kk$ such that
	\[
	(U T_i)^{\tau_i} \xrightarrow{(\tau_i)_{U T_i}} 
	U T_i \xrightarrow{\eta_i}
	U' T'_i \xrightarrow{u't' \mapsto \varepsilon(u') \iota_i(t')} 
	UT_i
	\]
	is an isomorphism of algebraic groups and thus the same holds for
	$\eta_i \circ (\tau_i)_{U T_i}$.
	In particular $\eta_i(U) = U'$ for all $i$.
	For $1 \leq i, j \leq n$, Lemma~\ref{Lem.Same_field_auto} applied to $\eta_i$
	and $\eta_j$ yield that $\tau_i = \tau_j$, since $\eta_i |_U = \eta |_U = \eta_j |_U$.
	Let $\tau \coloneqq \tau_1 = \ldots = \tau_n$.
	Hence,
	\[
	H^{\tau} \xrightarrow{\tau_{H}} 
	H \xrightarrow{\eta} H'
	\]
	restricts to a morphism of algebraic groups on $(U T_i)^{\tau}$ for all $i = 1, \ldots, n$.
	Since $H^\tau = (UT_1)^\tau \cdots (UT_n)^\tau$, the composition
	$\eta \circ \tau_H$ is a homomorphism of algebraic groups (see~\cite[Lemma~5.9]{LiReUr2023Characterization-o}).
	This implies the result.
\end{proof}

The following proposition essentially follows from \cite[Théorème~(A)]{BoTi1973Homomorphismes-abs}:

\begin{proposition}
	\label{Prop.Isom_of_simple_alg_grps}
	Let $\theta \colon G \to G'$ be a group isomorphism, where $G$, $G'$ are connected algebraic groups
	that are simple  (i.e.~non-commutative and every proper closed normal subgroup is finite and central). Then there
	exists a field automorphism $\tau$ of $\kk$ such that $\theta \circ \tau_G \colon G^\tau \to G'$ is
	an isomorphism of algebraic groups.
\end{proposition}

For the proof we will use the following lemma, which follows from the fact that every
commutative algebraic group is isomorphic to $\GG_m^r \times \GG_a^s \times F$ for $r, s \geq 0$ and
a finite commutative group $F$, see~\cite[Theorem~15.5]{Hu1975Linear-algebraic-g}.

\begin{lemma}
	\label{Lem.divisible_elements}
	If $H$ is a commutative algebraic group,
	then $H^\circ$ consists of those elements $h \in H$ that are divisible in $H$. \qed
\end{lemma}

\begin{proof}[Proof of Proposition~\ref{Prop.Isom_of_simple_alg_grps}]
	Let $Z$ and $Z'$ be the centers of $G$ and $G'$, respectively, and let $\pi \colon G \to G/Z$ and $\pi' \colon G' \to G'/Z'$
	be the canonical projections. Then there is a group
	isomorphism $\bar{\theta} \colon G/Z \to G'/Z'$ with $\bar{\theta} \circ \pi = \pi' \circ \theta$. 
	Using \cite[Théorème~(A)]{BoTi1973Homomorphismes-abs},
	there is a field automorphism $\tau$ of $\kk$ such that $\bar{\theta} \circ \tau_{G/Z} \colon (G/Z)^\tau \to G'/Z'$
	is an isomorphism of algebraic groups. 
	
	Let $U \subseteq G$ be a closed 
	subgroup that is isomorphic to $\GG_a$. Then $\GG_a \simeq \bar{\theta}(\pi(U))$
	is a closed subgroup of $G'/Z'$ and
	\[
	\theta(UZ) = \theta(\pi^{-1}(\pi(U))) = (\pi')^{-1}(\bar{\theta}(\pi(U))) = U'Z'
	\]
	for some closed subgroup $\GG_a \simeq U' \subseteq G'$. By Lemma~\ref{Lem.divisible_elements}
	it follows that $U = (UZ)^\circ$ consists of the divisible elements of $UZ$. Hence, $\theta(U)$
	consists of the divisible elements in $\theta(UZ) = U'Z'$, i.e.~$\theta(U) = (U'Z')^\circ = U'$. Thus
	$(\theta|_U) \circ \tau_{U} \colon U^\tau \to U'$ is an isomorphism of algebraic groups according
	to the following commutative diagram
	\[
	\xymatrix{
		U^\tau \ar[r]^-{\tau_U} \ar[d]_-{\pi^\tau |_{U^\tau}} & U \ar[r]^-{\theta|_U}
		\ar[d]^-{\pi|_U} & U' \ar[d]^-{\pi' |_{U'}} \\
		\pi(U)^\tau \ar[r]^-{\tau_{\pi(U)}} & \pi(U) \ar[r]^-{\bar{\theta}|_U} & \pi'(U')
	}
	\]
	since the vertical arrows are isomorphisms of algebraic groups and the composition of
	the bottom horizontal arrows as well. Since the closed subgroups that are isomorphic to $\GG_a$ generate the group
	$G$ (see~\cite[Theorem~27.5]{Hu1975Linear-algebraic-g}) 
	it follows from~\cite[Lemma~5.9]{LiReUr2023Characterization-o} that $\theta \circ \tau_G \colon G^\tau \to G'$
	is an isomorphism of algebraic groups.
\end{proof}

\section{Abstract group isomorphisms of automorphism groups}
\label{Sect.Preserving_alg_subgroups}

The aim of this section is to show that in certain situations
an abstract group isomorphism $\Aut(X) \to \Aut(Y)$ (where $X, Y$ are irreducible and affine)  
preserves certain algebraic subgroups.
Our main result in this direction is Theorem~\ref{Prop.Preserving_alg_groups}.

\begin{proposition}
	\label{Lem.Image_rootsubgrp}
	Let $\kk$ be uncountable, let $X, Y$ be irreducible affine varieties and let 
	$\theta \colon \Aut(X) \to \Aut(Y)$ be a group isomorphism.
	Assume that $H \subseteq \Aut(X)$ is a connected algebraic subgroup and let $V \subseteq \Aut(X)$ 
	be a non-trivial generalized $H$-root subgroup.
	If
	\begin{enumerate}[label=\alph*), wide=0pt, leftmargin=*]
		\item \label{Lem.Image_rootsubgrp1} $\theta(H)$ is an algebraic subgroup of $\Aut(Y)$ or
		\item \label{Lem.Image_rootsubgrp2} $H$ has a dense open orbit in $X$,
	\end{enumerate}
	then $\theta(V) \subseteq \Aut(Y)$ is a non-trivial $\theta(H)$-root subgroup
	of $\Aut(Y)$ of the same dimension as $V$.
\end{proposition}

The  lemma above is the key ingredient 
in the proof of Proposition~\ref{Prop.Centralizer_of_G_in_spherical} and in the proof of 
Theorem~\ref{Prop.Preserving_alg_groups} (which itself constitutes a crucial step in the
proof of Theorem~\ref{mainthm.A}).

\begin{proof}[Proof of Proposition~\ref{Lem.Image_rootsubgrp}]
	We establish first the statement in case $\dim V = 1$. For this, we establish the following:
	
	\begin{claim}
		\label{Claim.thetaU_closed}
		The subgroup $\theta(V)$ is closed in $\Aut(Y)$.
	\end{claim}

	Indeed, in case~\ref{Lem.Image_rootsubgrp1}, 
	fix some $v \in V \setminus \{ \id_X \}$. Then $H \to V \setminus \{ \id_X \}$, 
	$h \mapsto h \circ v \circ h^{-1}$ is surjective, since $H$ does not commute with $V$, and the map
	\begin{equation}
		\label{Eq.Surjection}
		\theta(H) \coprod \{ \ast \} \to \theta(V) \, , \quad
		g \mapsto
		\left\{\begin{array}{rl}
			g \circ \theta(v) \circ g^{-1} & \textrm{if $g \in \theta(H)$} \\
			\id_{Y} & \textrm{if $g = \ast$}
		\end{array} \right.
	\end{equation}
	is surjective as well. Since $\theta(H)$ is an algebraic subgroup of $\Aut(Y)$,
	\eqref{Eq.Surjection} is a morphism of a
	variety into $\Aut(Y)$ with image $\theta(V)$.
	Hence $\theta(V)$ is
	an algebraic subgroup of $\Aut(Y)$, see 
	\cite[Theorem~2.9]{KrReSa2021Is-the-affine-spac}. In particular, $\theta(V)$ is closed
	in $\Aut(Y)$. In case~\ref{Lem.Image_rootsubgrp2}, the claim follows from
	Lemma~\ref{Lem.Representation_of_on-parameter-subgrp}.
	
	\medskip
	
	Since $V$ is a non-trivial $H$-root subgroup and since $H$ is connected,
	there exists a closed subgroup $\GG_m \simeq D \subseteq H$  that acts by conjugation on $V$ with two orbits. 
	Thus $\theta(D)$ acts with two orbits on $\theta(V)$ by conjugation. 
	Let us write
	\[
	\overline{\theta(D)}^\circ = \bigcup_{i=1}^\infty K_i \, ,
	\]
	where $K_i$ is a commutative connected algebraic subgroup of $\Aut(Y)$ for all $i \geq 1$.

	By Claim~\ref{Claim.thetaU_closed} we know that $\theta(V)$
	is closed in $\Aut(Y)$ and by Corollary \ref{Cor.UnipotentToUnipotent}   $\theta(V)$ consists of unipotent elements.
	In particular $\theta(V)$ is connected and by Corollary~\ref{cor.CRX} we may write
	\[
	\theta(V) = \bigcup_{i=1}^\infty U_i
	\]
	where $U_i \simeq \GG_a^{k_i}$ for all $i \geq 1$. 
	Let $U_i'$ be the closure of the group generated by all $K_i$-conjugates
	of all elements of $U_i$. Then $U_i'$ is a closed subgroup of $U_{j(i)}$ for some 
	$j(i) \geq 1$. Hence we may replace every $U_i$ by $U_i'$ and assume that $K_i$ normalizes $U_i$.
	
	Let $K \subseteq \overline{\theta(D)}^\circ$ be the subgroup of elements that do commute with all 
	elements in $\theta(V)$. 
	
	\begin{claim}
		\label{Claim.Stabilizer}
		$\overline{\theta(D)}^\circ/K$ acts on $\theta(V)$ by conjugation and the stabilizer
		of each element in $\theta(V) \setminus \{ \id_Y \}$ by this action is trivial. 
	\end{claim}

	Indeed, let $g \in \overline{\theta(D)}^\circ$
	with $g \circ \theta(v) \circ g^{-1} = \theta(v)$ for some $\id_X \neq v \in V$. Then 
	$\theta^{-1}(g) \circ  v \circ \theta^{-1}(g^{-1}) = v$ and
	hence, $a \cdot v$ commutes with $\theta^{-1}(g)$ for all $a \in \ZZ$. As $\GG_a \simeq V$ this
	implies that $\theta^{-1}(g)$ commutes with all elements of $V$. Hence $g$ commutes with all
	elements of $\theta(V)$, i.e.~$g \in K$. This proves Claim~\ref{Claim.Stabilizer}.
	
	\medskip
	
	If $K_i / (K \cap K_i)$ contains a copy of $\GG_a$ for some $i \geq 1$, then by the Lie-Kolchin theorem
	applied to the action by conjugation of this copy on $U_i$  (see e.g.~\cite[Theorem~17.5]{Hu1975Linear-algebraic-g}),
	this copy would be contained in the stabilizer of some non-trivial element of $U_i$, contradiction to 
	Claim~\ref{Claim.Stabilizer}. Hence for all $i \geq 1$ there exists $d_i \geq 0$ such that
	$K_i / (K \cap K_i) \simeq \GG_m^{d_i}$.
	Since $\theta(D)$ acts with two orbits on $\theta(V)$ by conjugation, it follows that $\overline{\theta(D)}^\circ$
	acts with at most countably many orbits on $\theta(V)$ by conjugation. As $\kk$ is uncountable, this implies that $\overline{\theta(D)}^\circ/K$ is non-trivial
	and thus  $d_i \geq 1$ for all $i$ that are big enough.
	
	\begin{claim}
		\label{Claim.d_i=1}
		We claim that $d_i = 1$ for all $i$ that are big enough.
	\end{claim}
	
	Indeed, if $d_i \geq 2$ for some $i$, then there exists a subgroup $F_i \subseteq K_i/(K \cap K_i)$ that 
	is isomorphic to $(\ZZ/2\ZZ)^2$. Let $\pi_i \colon K_i \to  K_i/ (K \cap K_i)$ be the canonical projection.
	Then 
	\[
	K \cap \pi_i^{-1}(F_i) = \set{ g \in \pi_i^{-1}(F_i) }{ \textrm{$g$ commutes with all elements of $\theta(V)$} } \, .
	\]
	Hence, $(\ZZ/2\ZZ)^2 \simeq \theta^{-1}(\pi_i^{-1}(F_i))/\theta^{-1}(K \cap \pi_i^{-1}(F_i))$ acts faithfully on $V \simeq \GG_a$, contradiction.
	
	\medskip	
	
	By Claim~\ref{Claim.d_i=1} we get that the ind-group 
	$\overline{\theta(D)}^\circ/K \simeq \bigcup_{i=1}^\infty K_i/ (K \cap K_i)$ is isomorphic to $\GG_m$.
	Since $\overline{\theta(D)}^\circ/K$ acts via conjugation with at most countably many orbits on $\theta(V)$,
	we conclude that $\theta(V) \simeq \GG_a$, as $\kk$ is uncountable, 
	see \cite[Lemma~1.3.1]{FuKr2018On-the-geometry-of}.
	
	\medskip
	
	Now assume that $V \subset \Aut(X)$ has arbitrary dimension. Then we may write $V$ as a product $V = V_1 \cdots V_n$
	in $\Aut(X)$ for $H$-roots subgroups $V_1, \ldots, V_n \subseteq \Aut(X)$
	and the multiplication map $V_1 \times \cdots \times V_n \to V$ is an isomorphism of algebraic groups.
	By the proof above $\theta(V_i) \subseteq \Aut(Y)$ is isomorphic to $\mathbb{G}_a$.  Hence, 
	$\theta(V) = \theta(V_1)  \cdots \theta(V_n)$ is a closed algebraic subgroup of $\Aut(Y)$ and the
	multiplication map $\theta(V_1) \times \cdots \times \theta(V_n) \to \theta(V)$ is an isomorphism of algebraic groups.
	Thus $\theta(V)$ is isomorphic to $V$ as an algebraic group.
\end{proof}

\begin{lemma}\label{lem.torus_to_closed_set_implies_torus}
	Assume $\kk$ is uncountable, let $X, Y$ be irreducible affine varieties, 
	let $\theta \colon \Aut(X) \to \Aut(Y)$ be a group isomorphism
	and  let $D \subset \Aut(X)$ be a diagonalizable subgroup such that $\theta(D) \subseteq \Aut(Y)$ is closed.  Then $\theta(D)^\circ$ is a torus of dimension $\leq \dim D$.
\end{lemma}
\begin{proof}
	Since $\theta(D) \subset \Aut(Y)$ is a closed subgroup, $\theta(D)^\circ \subset \Aut(Y)$ is a closed connected subgroup. If $\theta(D)^\circ$ would contain a
	non-trivial unipotent element, then $\theta^{-1}(\theta(D)^\circ) \subset D$  would also contain a non-trivial unipotent element by Corollary \ref{Cor.UnipotentToUnipotent} which is not the case (here we use that $\kk$ is uncountable). Hence, every unipotent element in $\theta(D)^\circ$ is trivial and 
	we conclude that $\theta(D)^\circ$ is an algebraic subtorus of $\Aut(Y)$ by Corollary~\ref{cor.CRX}. 
	Since 
	the $p$-torsion of $\theta(D)$ is isomorphic to $(\mathbb{Z}/p\mathbb{Z})^{\dim D}$ 
	for every prime $p$, we get $\dim \theta(D)^\circ \le \dim D$. 
\end{proof}

\begin{proposition}	\label{Prop.Centralizer_of_G_in_spherical}
	Assume $\kk$ is uncountable.
	Let $X$ be $G$-spherical for a connected reductive algebraic subgroup $G$ of $\Aut(X)$ 
	and assume that $X$ is not isomorphic to an algebraic torus.
	Moreover, let $\theta \colon \Aut(X) \to \Aut(Y)$ be a group isomorphism, where $Y$ is irreducible and affine.
	Then $\theta(\Cent_{\Aut(X)}(G)^\circ)$ is a closed algebraic torus in $\Aut(Y)$ of the same dimension 
	than $\Cent_{\Aut(X)}(G)^\circ$.
	Moreover, 
	\[
		\theta(\Cent_{\Aut(X)}(G)^\circ) = \Cent_{\Aut(Y)}(\theta(G))^\circ \, .
	\]
\end{proposition}	

\begin{proof}
	Let $D \coloneqq \Cent_{\Aut(X)}(G)$. By Corollary~\ref{Cor.Knop} we get that $D$
	is a diagonalizable algebraic subgroup of $\Aut(X)$. Denote $r \coloneqq \dim D$.
	Let $B$ be a Borel subgroup of $G$. In case $B$ is not a torus, let $V = R_u(B)$ be the unipotent radical of $B$ 
	and otherwise, let $V$ be any $B$-root subgroup of $\Aut(X)$ (which exists, since $X$ is not isomorphic to a torus).
	Note that in the first case $BV = B$, whereas in the second case $D = B = G$ (due to Proposition~\ref{Lem.Centralizer_open_dense_orbit}\eqref{Lem.Centralizer_open_dense_orbit1}).
	Due to Proposition~\ref{Prop.Existence_of_root_subgroups}\eqref{Prop.Existence_of_root_subgroups1}
	applied to the solvable algebraic group $H \coloneqq DBV$
	there exist $DBV$-root subgroups
	$U_1, \ldots, U_r \subseteq \Aut(X)$ of linearly independent $D$-weights $\lambda_1 , \ldots, \lambda_r \colon D \to \GG_m$.
	For all $i=1, \ldots, r$, let
	\begin{eqnarray*}
		D_i &\coloneqq& \bigcap_{j \neq i} \ker(\lambda_j) = 
		\set{d \in D}{\textrm{$d$ commutes with all $u_j \in U_j$ for all $j \neq i$}} \\
		K &\coloneqq& \bigcap_{j = 1}^r \ker(\lambda_j) = \bigcap_{j=1}^r D_j \, .
	\end{eqnarray*}
	Note that $\dim D_i = 1$ for all $i$ and that $K$ is finite by Lemma~\ref{lem.Simultaneous_kernels}. 

	Since $D_i$ is the centralizer of $G$ and $U_1, \ldots, U_{i-1}, U_{i+1}, \ldots, U_r$ in $\Aut(X)$,
	it follows that $\theta(D_i)$ is closed in $\Aut(Y)$. Since $\dim D_i =1$, 
	we get
	\[	
	\theta(D_i)^\circ \simeq \GG_m \quad \textrm{for all $i = 1, \ldots, r$}	
	\]
	 by Lemma \ref{lem.torus_to_closed_set_implies_torus} 
	(here we use that $\theta(D_i)^\circ$ is infinite, since $\kk$ is uncountable).
	
	\begin{claim}
		\label{Claim.theta(D)circ}
		$\theta(D)^\circ$ is a closed torus of dimension $r$ in $\Aut(Y)$.
	\end{claim}

	Indeed, by Lemma~\ref{lem.torus_to_closed_set_implies_torus} it follows that $\theta(D)^\circ$ is a closed subtorus
	in $\Aut(Y)$ of dimension $\leq r$, as $\theta(D) = \Cent_{\Aut(Y)}(\theta(G))$
	is closed in $\Aut(Y)$. Consider the homomorphism of algebraic groups
	\[
	\eta \colon D_1 \times \cdots \times D_r \to D \, , \quad (d_1, \ldots, d_r) \mapsto d_1 \circ \cdots \circ d_r \, .
	\]
	If $(d_1, \ldots, d_r) \in D_1 \times \cdots \times D_r$ and $d_1 \circ \cdots \circ d_r = \id_X$, then
	$d_i$ commutes with all elements of $U_i$ and hence $d_i \in K$ for all $i$. 
	Thus,  the homomorphism
	$\eta$ has a finite kernel. Since $\theta(D_i)^\circ \simeq \GG_m$
	for all $i = 1, \ldots, r$,
	\[
	\rho \colon \theta(D_1)^\circ \times \cdots \times \theta(D_r)^\circ \to \theta(D)^\circ \, ,
	\quad  (t_1, \ldots, t_r) \mapsto t_1 \circ \cdots \circ t_r
	\]
	is a homomorphism of algebraic groups with a finite kernel and therefore $\dim \theta(D)^\circ \geq r$.
	This establishes Claim~\ref{Claim.theta(D)circ}.
	
	\begin{claim}
		\label{Claim.theta(Dcirc)}
		$\theta(D^\circ) = \theta(D)^\circ$
	\end{claim}
	
	Indeed, by Proposition~\ref{Lem.Image_rootsubgrp} applied to the connected algebraic subgroup $(DB)^\circ$ and the $(DB)^\circ$-root-subgroup $U_i$
	it follows that $\theta(U_i)$ is a $\theta((DB)^\circ)$-root subgroup of $\Aut(Y)$ for all $i$.
	The ind-group $\theta(D)$ acts by conjugation on $\theta(U_i) \simeq \GG_a$
	and induces thus a homomorphism of ind-groups $\xi_i \colon \theta(D) \to \Aut_{\grp}(\theta(U_i)) \simeq \GG_m$.
	Note that the kernel of the ind-group homomorphism
	\[
	\xi \coloneqq (\xi_1, \ldots, \xi_r) \colon \theta(D) \to \GG_m^r
	\]	
	is equal to the finite group $\theta(K)$. By Claim~\ref{Claim.theta(D)circ}, the restriction 
	$\xi |_{\theta(D)^\circ} \colon \theta(D)^\circ \to \GG_m^r$ is a homomorphism of 
	algebraic groups which is surjective, since $\theta(D)^\circ \simeq \GG_m^r$ and $\xi |_{\theta(D)^\circ}$ has a finite kernel. 
	Now, if $\theta(d) \in \theta(D)$, then there exists $d_0 \in D$ with $\theta(d_0) \in \theta(D)^\circ$ such that
	$\xi(\theta(d)) = \xi(\theta(d_0))$. In particular, there exists $k \in K$ with $\theta(d) = \theta(d_0) \theta(k)$. As
	$K$ is finite, we get that $\theta(D)^\circ$ has finite index in $\theta(D)$. Thus $\theta(D)$ is a closed diagonalizable algebraic subgroup
	of $\Aut(Y)$. As $D^\circ$ is the subgroup of divisible elements of $D$ (see Lemma~\ref{Lem.divisible_elements}), $\theta(D^\circ)$ is the subgroup
	of divisible elements in $\theta(D)$ and hence we get $\theta(D)^\circ = \theta(D^\circ)$.
	This proves Claim~\ref{Claim.theta(Dcirc)}.
	
	\medskip
	
	The Claims~\ref{Claim.theta(D)circ} and~\ref{Claim.theta(Dcirc)} imply thus the statement of the proposition.
\end{proof}

\begin{proposition}
	\label{Prop.Algebraic_homomorphism_on_torus}
	Let $\theta \colon \Aut(X) \to \Aut(Y)$ be an isomorphism of groups, where $X, Y$ are irreducible affine varieties.
	Assume that there is an algebraic torus $T \subseteq \Aut(X)$ of dimension $n$ such that
	$\theta(T) \subseteq \Aut(Y)$ is an algebraic subtorus. 
	Assume further that there are $T$-root subgroups
	$U_1, \ldots, U_{n}$ in $\Aut(X)$ of  linearly independent $T$-weights 
	$\lambda_1, \ldots, \lambda_{n} \colon T \to \GG_m$. Then
	there exist one-dimensional subtori $T_1, \ldots, T_n \subseteq T$ such that $T  = T_1 \cdots T_n$
	and $\theta(T_i)$ is a closed one-dimensional subtorus in $\Aut(Y)$ for all $i = 1, \ldots, n$.
\end{proposition}

\begin{proof}
	For all $i=1, \ldots, n$, let $T_i = S_i^\circ$, where
	\begin{equation}
		\label{Eq.Def_Si}
		S_i \coloneqq \bigcap_{j\neq i} \ker(\lambda_j) = 
		\set{ t \in T}{\textrm{$t$ commutes with all $u_j \in U_j$ for all $j \neq i$}} \, .
	\end{equation}
	Since the $T$-weights $\lambda_1, \ldots, \lambda_n$ are linearly independent, it follows that
	$\dim T_i = 1$ for all $i$ and the intersection $K = \bigcap_{i=1}^n \ker(\lambda_i)$ is finite, 
	see Lemma~\ref{lem.Simultaneous_kernels}. Note that the kernel of the multiplication map
	$T_1 \times \cdots \times T_n \to T$ lies in the finite group $K^n$ and thus $T_1 \cdots T_n = T$.
	By Lemma~\ref{Lem.divisible_elements}, $T_i$ consists of those elements in $S_i$	
	that are divisible in $S_i$. Hence, $\theta(T_i)$ consists of those elements in $\theta(S_i)$
	that are divisible in $\theta(S_i)$. However,
	since $\theta(T) \subset \Aut(Y)$ is a closed subtorus, the subgroup 
	$\theta(S_i) \subset \theta(T)$ is closed in $\Aut(Y)$ (see~\eqref{Eq.Def_Si}) and thus it is diagonalizable.
	By Lemma~\ref{Lem.divisible_elements}, 
	$\theta(T_i) = \theta(S_i)^\circ$ is a closed torus in $\Aut(Y)$. Since the $p$-torsion subgroup
	of $\theta(T_i)$ is isomorphic to $\ZZ/p \ZZ$ for all primes $p$, it follows that $\dim \theta(T_i) = 1$
	for all $i =1, \ldots, n$.
\end{proof}

Now, we come to the main result of this section.

\begin{theorem}\label{Prop.Preserving_alg_groups}
	Let $\kk$ be uncountable.
	Let $X$ be a $G$-spherical affine variety different from an algebraic torus for a connected reductive algebraic group $G$,
	let $B \subseteq G$ be a Borel subgroup and let 
	$\theta \colon \Aut(X) \to \Aut(Y)$ be a group isomorphism (where $Y$ is irreducible and affine).
	Then there exists a field automorphism $\tau$ of $\kk$ such that:
	\begin{enumerate}[label=\alph*), wide=0pt, leftmargin=*]
		\item \label{Prop.Preserving_alg_groups1} $\theta(G)$ is an algebraic subgroup of $\Aut(Y)$ and 
		$\theta \circ \tau_G \colon G^{\tau} \to \theta(G)$ is 
		an isomorphism of algebraic groups.
		\item \label{Prop.Preserving_alg_groups2}
		Let $H$ be a closed connected algebraic subgroup of $G$ and let
		$U^\tau \subseteq \Aut(X^\tau)$ 
		be a non-trivial $H^\tau$-root subgroup of weight $\chi \colon H^\tau \to \GG_m$. Then 
		$\theta(U)$ is a non-trivial $\theta(H)$-root subgroup of weight 
		\[
		\chi' \colon \theta(H) 
		\xrightarrow{ (\tau_H)^{-1} \circ (\theta |_H)^{-1}} H^{\tau} \xrightarrow{\chi} \GG_m \, 
		\]
		and $\theta |_U \circ \tau_U \colon  U^\tau \to \theta(U)$ is an isomorphism of algebraic groups.
		\item \label{Prop.Preserving_alg_groups3}  
		Let $H$ be a closed connected subgroup of $G$. Then
		for every $d \geq 1$
		we have a bijective correspondence 
		\[
			U \mapsto \theta(U)
		\] between the non-trivial  generalized $H$-root subgroups in $\Aut(X)$
		of dimension $d$ and the non-trivial generalized
		$\theta(H)$-root subgroups of $\Aut(Y)$ of dimension $d$.
		\item \label{Prop.Preserving_alg_groups4}   
		Let $H$ be a closed connected subgroup of $G$ such that all
		$H$-root subgroups of $\Aut(X)$ are non-trivial. Then all $\theta(H)$-root subgroups
		of $\Aut(Y)$ are non-trivial.
	\end{enumerate}
\end{theorem}

\begin{proof}
	We may assume $\dim G \geq 1$.
	
	\ref{Prop.Preserving_alg_groups1}: 
	Let $D \coloneqq \Cent_{\Aut(X)}(G)^\circ$.  By Corollary~\ref{Cor.Knop} $D$ is a torus.
	Then $G' \coloneqq D G$ is still a connected reductive algebraic subgroup
	of $\Aut(X)$
	and $X$ is $G'$-spherical. Hence, we may replace $G$ by $G'$ and assume that $D$ is a subgroup of $G$.
	
	As $G$ is connected and reductive, $[G, G]$ is semi-simple or trivial, and
	the neutral  component of the center $Z(G)^\circ$ of $G$ is a central torus in $G$ such that $G = Z(G)^\circ [G, G]$
	and the intersection $Z(G)^\circ \cap [G, G]$ is finite, see
	\cite[IV, Proposition 14.2]{Bo1991Linear-algebraic-g}. Since $D$ is contained in $G$, 
	it follows that 
	$D = Z(G)^\circ$. By Proposition~\ref{Prop.Centralizer_of_G_in_spherical} it follows that
	$\theta(D)$ is a closed algebraic subtorus of $\Aut(Y)$ of the same dimension than $D$.
	
	\begin{claim}
		\label{Claim.thetaG_closed}
		The image $\theta(G)$ is a closed algebraic subgroup of $\Aut(Y)$.
	\end{claim}
	
	Indeed, if $G$ is commutative (i.e.~it is a torus), then $G = D$ and we already proved Claim~\ref{Claim.thetaG_closed} in the last paragraph.
	Hence, we may assume that $G$ is non-commutative.
	Write $[G, G] = G_1 \cdots G_n$, where $n \geq 1$ and each $G_i$ is a simple
	algebraic subgroup of $[G, G]$ \cite[Theorem~27.5]{Hu1975Linear-algebraic-g} such that
	the multiplication map $G_1 \times \cdots \times G_n \to [G, G]$ is a central isogeny and
	$G_1, \ldots, G_n$ commute pairwise. Then 
	the multiplication map $\rho \colon D \times G_1 \times \cdots \times G_n \to G$ is a central isogeny as well.
	Let $B_i \coloneqq \pi_i(\rho^{-1}(B))^\circ$ where $\pi_i \colon D \times G_1 \times \cdots \times G_n \to G_i$
	denotes the canonical projection.
	Then $B_i \subseteq G$ is a Borel subgroup of $G_i$, see~\cite[Corollary~21.3C]{Hu1975Linear-algebraic-g}.

	Let $\GG_a\simeq U_i \subseteq B_i$ be a $B_i$-root subgroup in the center of the unipotent radical of $B_i$. 
	Then $U_i$ is a non-trivial $B_i$-root subgroup. Note that the closure $H_i$ of the group generated
	by all $g_iU_ig_i^{-1}$, $g_i \in G_i$ is a closed normal infinite algebraic subgroup of $G_i$. As $G_i$ is simple, we get that
	$H_i = G_i$. 
	Using \cite[Proposition~7.5]{Hu1975Linear-algebraic-g}, there exist $g_{i, 1}, \ldots, g_{i, m_i} \in G_i$ with
	\begin{equation}
		\label{Eq.Composition_of_G_i}
	g_{i, 1} U_i g_{i, 1}^{-1} \cdots g_{i, m_i} U_i g_{i, m_i}^{-1} = G_i \, .
	\end{equation}
	Note that $\rho(D \times B_1 \times \cdots \times B_n) = DB_1 \cdots B_n$ 
	is a Borel subgroup of $G = D [G, G]$ (see \cite[Corollary~21.3C]{Hu1975Linear-algebraic-g}) 
	that normalizes, but not centralizes $U_i$
	for all $i=1, \ldots, n$. Since $B$ is a Borel subgroup 
	contained in $(D \pi_1( \rho^{-1}(B)) \cdots \pi_n( \rho^{-1}(B)))^\circ = DB_1 \cdots B_n$, it follows that
	\[
	B = DB_1 \cdots B_n \, .
	\]
	By Proposition~\ref{Lem.Image_rootsubgrp} applied to $B$ it follows that
	$\theta(U_i)$ is a non-trivial $\theta(B)$-root subgroup of $\Aut(Y)$ for all $i$.
	Thus $\theta(g_{i, j}U_i g_{i, j}^{-1})$ is isomorphic to $\GG_a$ for all $i, j$. Hence it follows from~\eqref{Eq.Composition_of_G_i} 
	that $\theta(G_i)$
	is a closed algebraic subgroup of $\Aut(Y)$ for all $i$. We conclude that $\theta(G) = \theta(D)\theta(G_1) \cdots \theta(G_n)$ is a closed algebraic subgroup of $\Aut(Y)$  as well.
	This finishes the proof of Claim~\ref{Claim.thetaG_closed}.
	
	\medskip
	
	By \cite[Corollary~29.5]{Hu1975Linear-algebraic-g} it follows that $G_i / Z(G_i)$ is simple as an abstract group,
	where $Z(G_i)$ denotes the center of $G$. Hence $\theta(G_i) / Z(\theta(G_i))$ is simple as an abstract group and thus
	$\theta(G_i)$ is a simple algebraic group. Thus, by Proposition~\ref{Prop.Isom_of_simple_alg_grps}
	there exists a field automorphism $\tau_i$ of $\kk$ such that
	\begin{equation}
		\label{Eq.Isom_G_i}
		\theta |_{G_i} \circ (\tau_i)_{G_i} \colon G_i^{\tau_i} \to \theta(G_i)
		\quad \textrm{is an isomorphism of algebraic groups}
	\end{equation}
	for $i =1, \ldots, n$. Let $T_i$ be a one-dimensional torus in $B_i$. 
	Hence, $\theta(T_i)$ is a closed subgroup in $\theta(G_i)$ isomorphic to $\GG_m$.
	Sine $T_i$ and $T_j$ commute for all $i, j$ and since $T_i$ commutes with $D$ for all $i$,
	there exists a maximal torus $T \subseteq B$ that contains $D$ and $T_i$ for all $i =1, \ldots, n$.  
	If $G$ is not a torus, let $V$ be a $B$-root subgroup of the center of $R_u(B)$
	and in case $G$ is a torus, let $V$ be any $B$-root subgroup of $\Aut(X)$. 
	Let $r = \dim D$. 
	By Proposition~\ref{Prop.Existence_of_root_subgroups}\eqref{Prop.Existence_of_root_subgroups1} applied
	to the solvable algebraic group $DV$ there exist
	$D$-root subgroups $U_1, \ldots, U_r$ such that the $D$-weights $\lambda_1, \ldots, \lambda_r \colon D \to \GG_m$
	are linearly independent. Due to Proposition~\ref{Prop.Algebraic_homomorphism_on_torus} applied to the torus $D$
	there exist one-dimensional tori $D_1, \ldots, D_r \subseteq D$ such that $D = D_1 \cdots D_r$, 
	$\theta(D_i)$ is closed in $\theta(D)$ and $\theta(D_i) \simeq \GG_m$ as algebraic groups for all 
	$i = 1, \ldots, r$.
	By Proposition~\ref{Prop.Existence_of_root_subgroups}\eqref{Prop.Existence_of_root_subgroups2} applied
	to the solvable algebraic group $BV$ and the tori $D_1, \ldots, D_r, T_1, \ldots, T_n$ inside $T \subseteq BV$  
	there exists a $T$-root subgroup $U \subseteq \Aut(X)$ such that non of the tori 
	$D_1, \ldots, D_r, T_1, \ldots, T_n$ commutes with $U$. By Proposition~\ref{Lem.Image_rootsubgrp}
	applied to one of the tori $D_1, \ldots, D_r, T_1, \ldots, T_n$
	it follows that $\theta(U)$ is closed in $\Aut(Y)$ and isomorphic to $\GG_a$. Let 
	$H \coloneqq UD_1 \cdots D_rT_1 \cdots T_n$.
	Then 
	\[
	\theta(H) = \theta(U) \theta(D_1) \cdots \theta(D_r) \theta(T_1) \cdots \theta(T_n) 
	\subseteq \Aut(Y)
	\]
	is a solvable algebraic subgroup of $\Aut(Y)$. Since $\theta(U) \simeq \GG_a$ and
	$\theta(D_i) \simeq \theta(T_j) \simeq \GG_m$ for all $i =1, \ldots, r$, $j = 1, \ldots, n$, 
	we may apply Lemma~\ref{Lem.Constructing_Iso_of_alg_grps} to $\theta |_H \colon H \to \theta(H)$
	in order to get a field automorphism $\tau$ of $\kk$ such that
	\begin{equation}
		\label{Eq.Isom_H}
		\theta |_H \circ \tau_H \colon H^\tau \to \theta(H)
		\quad \textrm{is an isomorphism of algebraic groups} \, .
	\end{equation}
	Using,~\eqref{Eq.Isom_G_i} and~\eqref{Eq.Isom_H}, Lemma~\ref{Lem.Same_field_auto} 
	implies that $\tau = \tau_i$ for all $i =1, \ldots, n$. Since $D = D_1 \cdots D_r \subseteq H$, 
	it follows again from~\eqref{Eq.Isom_H} that $\theta|_G \circ \tau_G \colon G^\tau \to \theta(G)$
	restricted to $D^\tau$ is a homomorphism of algebraic groups and from~\eqref{Eq.Isom_G_i} the
	same holds for $G_i^\tau$, $i=1, \ldots, n$. Using that $G^\tau = D^\tau G_1^\tau \cdots G_n^\tau$,
	\cite[Lemma~5.9]{LiReUr2023Characterization-o} gives that 
	$\theta|_G \circ \tau_G \colon G^\tau \to \theta(G)$ is an isomorphism of algebraic groups.
	
	\ref{Prop.Preserving_alg_groups2}:
	Using~\ref{Prop.Preserving_alg_groups1} it follows that $\theta(H)$ is an algebraic subgroup 
	in $\Aut(Y)$ and 
	\[
	H^\tau \xrightarrow{\tau_H} H \xrightarrow{\theta|_H} \theta(H)
	\] 
	is an isomorphism of algebraic groups. 
	As $U^\tau$ is a non-trivial $H^\tau$-subgroup,
	Proposition~\ref{Lem.Image_rootsubgrp} implies that 
	$\theta(U)$ is a non-trivial $\theta(H)$-root subgroup of $\Aut(Y)$.
	Let $T \subseteq H$ be any one-dimenional torus that does not commute with $U$. 
	Then $\theta(T)$ is a closed subgroup of $\theta(H)$
	and $\theta(T) \simeq \GG_m$. We apply Lemma~\ref{Lem.Constructing_Iso_of_alg_grps}
	to $\theta |_{UT} \colon UT \to \theta(UT)$ in order to get a 
	field automorphism $\sigma$ of $\kk$ such that 
	$\theta |_{UT} \circ\sigma_{UT} \colon (UT)^\sigma \to \theta(UT)$ is an isomorphism of algebraic
	groups. Lemma~\ref{Lem.Same_field_auto} yields $\tau = \sigma$ and hence
	$\theta|_U \circ \tau_U \colon U^\tau \to \theta(U)$ is an isomorphism of 
	algebraic groups. 
	
	Let $\varepsilon \colon \GG_a \to U^\tau$ be an isomorphism of algebraic groups. 
	For $h \in H$, we denote $h^\tau \coloneqq (\tau_H)^{-1}(h) \in H^\tau$.
	By definition, we have
	\[
	\varepsilon( \chi(h^{\tau}) \cdot s) = h^{\tau} \circ \varepsilon(s)  \circ (h^\tau)^{-1} \quad \textrm{for all $s \in \GG_a$ 
		and all $h^{\tau} \in H^\tau$} \, .
	\]
	Hence, $\varepsilon' \coloneqq \theta |_U \circ \tau_U \circ \varepsilon \coloneqq \GG_a \to \theta(U)$ is an
	isomorphism of algebraic groups that satisfies for all $h \in H$ and $s \in \GG_a$:
	\begin{eqnarray*}
		\varepsilon'(\chi'(\theta(h)) \cdot s) &=& (\theta |_U \circ \tau_U \circ \varepsilon)(\chi(h^\tau)\cdot s) \\
		&=& (\theta |_U \circ \tau_{U})(h^{\tau} \circ \varepsilon(s)  \circ (h^\tau)^{-1}) \\
		&=& \theta(h \circ \tau_U(\varepsilon(s)) \circ h^{-1}) ) \\
		&=& \theta(h) \circ \varepsilon'(s) \circ \theta(h)^{-1} \, ,
	\end{eqnarray*}
	i.e. $\theta(U)$ is a $\theta(H)$-root subgroup of weight $\chi'$.
	
	\ref{Prop.Preserving_alg_groups3}: Let $V \subseteq \Aut(Y)$ be a non-trivial
	generalized $\theta(H)$-root subgroup of dimension $d \geq 1$. 
	By Proposition~\ref{Lem.Image_rootsubgrp} applied to the group
	isomorphism $\theta^{-1} \colon \Aut(Y) \to \Aut(X)$ and the connected algebraic subgroup $\theta(H)$ of $\Aut(Y)$, 
	we get that
	$\theta^{-1}(V)$ is a non-trivial generalized $H$-root subgroup in $\Aut(X)$ 
	of dimension $d$. 
	
	On the other hand, let $U$ be a non-trivial generalized $H$-root subgroup of dimension $d$.
	Let $U_1, \ldots, U_d \subseteq U$ be $H$-root subgroups with
	$U = U_1 \cdots U_d$. As $U_i$
	is a non-trivial $H$-root subgroup of $\Aut(X)$, by~\ref{Prop.Preserving_alg_groups2}
	we know that $\theta(U)$ is a non-trivial $\theta(H)$-root subgroup of $\Aut(Y)$ contained in 
	$\theta(U)$. Since multiplication
	\[
		\theta(U_1) \times \cdots \times \theta(U_d) \to  \Aut(Y)
	\]
	gives an injective ind-group homomorphism of an algebraic group into an 
	ind-group with image equal to $\theta(U)$ we get that
	$\theta(U)$ is a non-trivial generalized $\theta(H)$-root subgroup of $\Aut(Y)$ of dimension $d$.
	This gives~\ref{Prop.Preserving_alg_groups3}.
	
	\ref{Prop.Preserving_alg_groups4}:
	Let $V$ be a trivial $\theta(H)$-root subgroup of $\Aut(Y)$.
	Then $\theta^{-1}(V)$ contains non-trivial unipotent elements that commute with $H$
	by Corollary~\ref{Cor.UnipotentToUnipotent}. The closure of the subgroup generated by 
	such a unipotent element in
	$\Aut(X)$ is a trivial $H$-root subgroup  and thus we arrive 
	at a contradiction.
\end{proof}

\begin{proposition}
	\label{Prop.Spherial}
	Let $\kk$ be uncountable.
	Let $X$ be a $G$-spherical affine variety for a connected reductive algebraic subgroup $G \subseteq \Aut(X)$,
	let $B \subseteq G$ be a Borel subgroup and let 
	$\theta \colon \Aut(X) \to \Aut(Y)$ be a group isomorphism (where $Y$ is irreducible and affine).
	If $X$ is not isomorphic to an algebraic torus, then $\theta(B)$ acts with a dense open orbit on $Y$.
\end{proposition}

We know from Theorem~\ref{Prop.Preserving_alg_groups}\ref{Prop.Preserving_alg_groups1} that $\theta(G)$
is a connected reductive algebraic subgroup of $\Aut(Y)$ of the same dimension than $G$ and $\theta(B)$
is a Borel subgroup of $\theta(G)$.
For the proof of Proposition~\ref{Prop.Spherial} we need the following result:

\begin{proposition}[{see~\cite[Proposition~7.3]{ReSa2021Characterizing-smo}}]
	\label{Prop.Characterization_Spherical}
	Let $X$ be an irreducible affine variety and let $H$ be a connected solvable algebraic 
	subgroup of $\Aut(X)$ that contains non-trivial unipotent elements. Then the following
	statements are equivalent:
	\begin{enumerate}[wide=0pt, leftmargin=*]
		\item $H$ acts with a dense open orbit on $X$;
		\item There exists a constant $C$ such that the dimension of every generalized 
		$H$-root subgroup of $\Aut(X)$ is bounded by $C$. \qed
	\end{enumerate}
\end{proposition}

\begin{proof}[Proof of Proposition~\ref{Prop.Spherial}]
	We distinguish two cases:
	
	\textbf{$G$ is not a torus}: In this case $B$ contains non-trivial unipotent elements
	and hence $\theta(B)$ as well by Corollary~\ref{Cor.UnipotentToUnipotent}.
	Thus we may apply Proposition~\ref{Prop.Characterization_Spherical}. 
	Note that $\Aut(X)$ contains no trivial $B$-root subgroup by Lemma~\ref{Lem.No_B-root subgroups}.
	The statement follows now from
	Theorem~\ref{Prop.Preserving_alg_groups}\ref{Prop.Preserving_alg_groups3}, \ref{Prop.Preserving_alg_groups4}.
	
	\textbf{$G$ is a torus}: Let $n = \dim G$. Since $X$ is not isomorphic to an algebraic torus,
	there exists a $G$-root subgroup $U$ in $\Aut(X)$, see 
	Lemma~\ref{Lem.Characterize_Torus}. Let $D \subseteq G$ be the subgroup
	of those elements in $G$ that commute with $U$. Then $H \coloneqq DU$ is a
	connected commutative algebraic subgroup of $\Aut(X)$ that acts with a dense open orbit on $X$
	by Lemma~\ref{Lem.Dense_open_DU-orbit}. Since $U$ is a non-trivial $G$-root subgroup
	and since $D \subseteq G$, it follows that $\theta(D)$ and $\theta(U)$ are closed algebraic subgroups
	of $\Aut(Y)$, isomorphic to $D$ and $U$, respectively 
	(see Theorem~\ref{Prop.Preserving_alg_groups}). Hence, $\theta(H) = \theta(D) \theta(U)$
	is an algebraic subgroup of $\Aut(Y)$ that is isomorphic to $DU$. By Proposition~\ref{Lem.Centralizer_open_dense_orbit}\eqref{Lem.Centralizer_open_dense_orbit1}
	it follows $\Cent_{\Aut(X)}(H) = H = DU$ and therefore $\Cent_{\Aut(Y)}(\theta(H)) = \theta(D) \theta(U)$.
	This implies that $U$ is the only trivial $H$-root subgroup in
	$\Aut(X)$ and $\theta(U)$ is the only trivial $\theta(H)$-root subgroup of $\Aut(Y)$.
	Hence, $U$ and $\theta(U)$ are also the only trivial generalized root subgroups of $\Aut(X)$ and $\Aut(Y)$
	with respect to $H$ and $\theta(H)$, respectively.
	Now, Proposition~\ref{Prop.Characterization_Spherical}
	in combination with Theorem~\ref{Prop.Preserving_alg_groups}\ref{Prop.Preserving_alg_groups3}
	applied to $H$ implies that $\theta(H)$ acts with a dense open orbit on $Y$ and thus
	$\dim Y = \dim \theta(H) = \dim H = \dim X = n$. The torus $\theta(B) = \theta(G)$ acts thus with a dense open orbit on $Y$.
\end{proof}

\section{Characterizing affine spherical varieties diffferent from a torus}
\label{Sec.CharSpherical_different_from_torus}

This section is devoted to the proof of 
Theorem~\ref{mainthm.A}. In fact, we will prove a
more general result. For its formulation we need the following notation:
If $X$ is an affine varitey endowed with a faithful action of an algebraic group $H$,
we denote
\[
D_H(X) = \set{ \lambda \in \frak{X}(H) }{ 
	\textrm{there exists an $H$-root subgroup of $\Aut(X)$ of weight $\lambda$}} \, .
\]
Note that in case $X$ is $G$-spherical and $B$ is a Borel subgroup of $G$, then 
the zero weight is not contained in $D_B(X)$ by Lemma~\ref{Lem.No_B-root subgroups}.

\begin{theorem}\label{thm.Ageneralization}
	Assume that $\kk$ is uncountable, let $X$, $Y$ be irreducible affine varieties, let
	$\theta \colon \Aut(X) \to \Aut(Y)$ be a group isomorphism
	and let $G \subseteq \Aut(X)$ be a connected reductive 
	algebraic subgroup. 
	Moreover, we fix a Borel subgroup $B \subseteq G$. 
	If $X$ is $G$-spherical and not isomorphic to an algebraic torus, then 
	there exists a field automorphism $\tau$ of $\kk$ such that:
	\begin{enumerate}[wide=0pt, leftmargin=*]
		\item \label{main1gen} 
		The image $\theta(G)$ is an algebraic subgroup in $\Aut(Y)$ and
		\[
		G^\tau \xrightarrow{\tau_G} G \xrightarrow{\theta |_G} \theta(G)
		\]
		is an isomorphism of algebraic groups.
		\item \label{main2gen}  The algebraic subgroup $\theta(B)$ of $\Aut(Y)$ acts 
		with a dense open orbit on $Y$.
		\item \label{main2.5gen}
		The isomorphism
		\[
		\vartheta \colon \frak{X}(\theta(B)) \to \frak{X}(B^\tau) \, , \quad 
		\lambda \mapsto \lambda \circ \theta |_B \circ \tau_B
		\]
		sends the set $D_{\theta(B)}(Y)$ onto the set $D_{B^\tau}(X^\tau)$.
		\item \label{main3gen}
		If $Y$ is normal, then $\vartheta \colon \frak{X}(\theta(B)) \to \frak{X}(B^\tau)$ from above
		sends the weight monoid $\Lambda^+_{\theta(B)}(Y)$ onto the weight
		monoid $\Lambda^+_{B^\tau}(X^\tau)$.
		\item \label{main4gen} If $X$ and $Y$ are smooth, 
		then there exists an isomorphism $\varphi \colon X^\tau \to Y$ such that
		we have the following commutative diagram
		\[
		\xymatrix@R=10pt{
			G^\tau \times X^\tau \ar[d] \ar[rr]^-{(\theta |_G \circ \tau_G) \times \varphi} 
			&& \theta(G)  \times Y \ar[d] \\
			X^\tau \ar[rr]^-{\varphi} && Y
		}
		\]
		where the vertical arrows are the action morphisms.
	\end{enumerate}
\end{theorem}

\begin{proof}[Proof of Theorem~\ref{thm.Ageneralization}]
	Let $\tau$ be the field automorphism of $\kk$ from Theorem~\ref{Prop.Preserving_alg_groups}.
	Then~\eqref{main1gen} and~\eqref{main2gen} follow from Theorem~\ref{Prop.Preserving_alg_groups}\ref{Prop.Preserving_alg_groups1} 
	and~Proposition~\ref{Prop.Spherial}, respectively.
	
	Next we note that $\theta |_B \circ \tau_B \colon B^\tau \to \theta(B)$
	is an isomorphism of algebraic groups. Hence
	\[
		\vartheta \colon
		\frak{X}(\theta(B)) \to \frak{X}(B^\tau) \, , \quad 
		\lambda \mapsto \lambda \circ \theta |_B \circ \tau_B
	\]
	is an isomorphism of groups. By Theorem~\ref{Prop.Preserving_alg_groups}\ref{Prop.Preserving_alg_groups2},
    \ref{Prop.Preserving_alg_groups3}, \ref{Prop.Preserving_alg_groups4} the isomorphism $\vartheta$ restricts to a bijection
    $D_{\theta(B)}(Y) \to D_{B^\tau}(X^\tau)$. Now~\eqref{main2.5gen} follows 
    and~\eqref{main3gen} follows from
    \cite[Corollary~8.4]{ReSa2021Characterizing-smo}.
    The last statement~\eqref{main4gen} follows from Losev's Theorem~\cite[Theorem~1.3]{Lo2009Proof-of-the-Knop-}.
\end{proof}

Now we show, how Theorem~\ref{thm.Ageneralization} implies Theorem~\ref{mainthm.A}:

\begin{proof}[Proof of Theorem~\ref{mainthm.A}]
	Denote by  $\tau$ the field automorphism of $\kk$ from Theorem~\ref{thm.Ageneralization}\eqref{main1gen}.
	Since $X$, $G$ and the action of $G$ on $X$ are defined over $\QQ$, it follows that 
	there exists an isomorphism of $\kk$-varieties $\varphi \colon X \to X^\tau$ and an isomorphism
	of algebraic groups $\vartheta \colon G \to G^\tau$ such that
	\[
		\varphi(g x) = \vartheta(g) \varphi(x) \quad \textrm{for all $g \in G$ and all $x \in X$}
	\]
	by Remark~\ref{Rem.Xtau_X_def_over_Q}. Hence, we may identify $X^\tau$, $G^\tau$
	and the action of $G^\tau$ on $X^\tau$ with $X$, $G$ and the action of $G$ on $X$, respectively.
	Moreover, if $B$ is a Borel subgroup of $G$ that is defined over $\QQ$, then $\tau_G \colon G \to G$
	maps $B$ onto itself.
	
	With these preliminary considerations, Theorem~\ref{mainthm.A} follows from
	Theorem~\ref{thm.Ageneralization}.
\end{proof}

\begin{proof}[Proof of Corollary~\ref{maincor.B}]
	Note that every toric variety is defined over $\QQ$ (up to an isomorphism) and that an affine toric variety
	is determined by its weight monoid. Hence the statement follows from Theorem~\ref{mainthm.A}.
\end{proof}

\begin{remark}
	\label{Rem.Both toric}
	In case $X$ and $Y$ are both affine toric varieties and $\Aut(X)$, $\Aut(Y)$ are isomorphic as groups,
	then $X$ and $Y$ are isomorphic as toric varieties. 
	
	Indeed, if $X$ is non-isomorphic to an algebraic torus, then
	$X$, $Y$ are isomorphic as varieties by Corollary~\ref{maincor.B}. If $X$ is isomorphic to an algebraic torus, then
	$\Aut(X)$ has no non-trivial unipotent element and hence
	$\Aut(Y)$ has no non-trivial unipotent element by Corollary~\ref{Cor.UnipotentToUnipotent}. 
	By Lemma~\ref{Lem.Characterize_Torus},
	$Y$ is isomorphic to an algebraic torus. Since 
	in this case
	the divisible elements in $\Aut(X)$ and $\Aut(Y)$ are exactly the maximal tori, we conclude again that $X$ and $Y$ 
	are isomorphic as varieties. Now the claim follows from Demushkin’s theorem \cite[Theorem~2]{De1982Combinatorial-inva}.
\end{remark}

\section{Characterizing a torus}
\label{sec.X_is_the_torus}

This section is devoted to the study of the case when $X$ is an algebraic torus and the proof
of Theorem~\ref{mainthm.C}. The following consequence of \cite[Theorem~A]{CaXi2018Algebraic-actions-} is 
the major ingredient:

\begin{theorem}
	\label{thm.CX}
	Let $\Gamma$ be a subgroup of finite index of $\SL_n(\ZZ)$ and 
	let $Z$ be an irreducible quasi-affine variety such that
	$\Gamma$ embeds into $\Aut(Z)$.
	Then $\dim Z \geq n$.	
\end{theorem}

\begin{proof}[Proof of Theorem~\ref{thm.CX}]
	As \cite[Theorem~A]{CaXi2018Algebraic-actions-} is formulated over the complex numbers
	$\CC$, we have to reduce to that case.
	Since $\SL_n(\ZZ)$ is finitely generated, $\Gamma$ is finitely generated 
	as well by Schreier‘s lemma; see e.g.~\cite[Lemma 4.2.1]{Se2003Permutation-group-}. 
	Hence, we may find a subfield $\kk_0$ of $\kk$ that is finitely generated
	over $\QQ$ such that there is an irreducible 
	quasi-affine variety $Z_0$, defined over the algebraic closure $\overline{\kk_0}$ with the following properties:
	\begin{itemize}
		\item $Z = Z_0 \times_{\overline{\kk_0}} \kk$ (note that $\overline{\kk_0}$ embeds into $\kk$, as $\kk$ is algebraically closed);
		\item $\Gamma$ embeds into $\Aut_{\overline{\kk_0}}(Z_0)$. 
	\end{itemize}
	As $\kk_0$ is finitely generated over $\QQ$, it embeds into $\CC$
	and therefore $\overline{\kk_0}$ embeds into $\CC$ as well.
	By \cite[Theorem~A]{CaXi2018Algebraic-actions-} we get thus
	$\dim_{\CC} Z_0 \times_{\overline{\kk_0}} \CC \geq n$. Since
	$\dim_{\kk} Z = \dim_{\overline{\kk_0}} Z_0 = \dim_{\CC} Z_0 \times_{\overline{\kk_0}} \CC$, 
	the statement follows.
\end{proof}

\begin{proof}[Proof of Theorem~\ref{mainthm.C}]
	Let $X = T$ be a torus of dimension $n$.
	We may assume $n \geq 1$.
	For every prime $p$, denote by $\Gamma_p$ the kernel of the natural homomorphism
	$\SL_n(\ZZ) \to \SL_n(\ZZ / p \ZZ)$. Note that $\Gamma_p$ has finite index in 
	$\SL_n(\ZZ)$ for all primes $p$.
	Since $\bigcap_{i} \Gamma_{p_i}$ is the trivial group for all infinite sequences of primes
	$p_1, p_2, \ldots$, for every finite subgroup 
	$F \subseteq \SL_n(\ZZ)$ there exists a prime  $p$ with
	\begin{equation}
		\label{Eq.Gamma_m_intersects_F_trivially}
		F \cap \Gamma_p = \{ I \} \, ,
	\end{equation}
	where $I$ denotes the identity matrix in $\SL_n(\ZZ)$.
	
	Let $\theta \colon \Aut(T) \to \Aut(Y)$ be a group isomorphism.
	As $T$ is its own centralizer in $\Aut(T)$ (see~Proposition~\ref{Lem.Centralizer_open_dense_orbit}\eqref{Lem.Centralizer_open_dense_orbit1}), it follows
	that $\theta(T)$ is closed in $\Aut(Y)$. Hence, $\theta(T)^\circ$
	is a closed subtorus in $\Aut(Y)$ of dimension $\leq n$ by Lemma~\ref{lem.torus_to_closed_set_implies_torus}.
	
	Consider the natural group isomorphism
	\[
	\begin{array}{rcl}
		\GL_n(\ZZ) &\to& \Aut_{\grp}(T) \, , \\
		(a_{ij})_{ij} &\mapsto& 
		\left( (t_1, \ldots, t_n) \mapsto
		\left( \prod_{i=1}^n t_i^{a_{1i}}, \ldots, \prod_{i=1}^n t_i^{a_{ni}} \right) \right) \, ,
	\end{array}
	\]
	where $\Aut_{\grp}(T)$ denotes the group of all algebraic group automorphisms of $T$. 
	Note that this homomorphism is induced by the action by conjugation of 
	$\Aut_{\grp}(T)$ on $T$ inside $\Aut(T)$.
	Using the isomorphism above, we may see $\SL_n(\ZZ)$ as a subgroup of $\Aut(T)$
	and this subgroup normalizes $T$. Hence, $\theta(\SL_n)$ normalizes $\theta(T)^\circ$.
	We distinguish three cases, depending on the dimension of $\theta(T)^\circ$:
	
	\begin{itemize}
		\item $\dim \theta(T)^\circ = n$: As $\dim Y \leq n$, $Y$ is an affine  $n$-dimensional toric variety. 
		By Theorem~\ref{Prop.Preserving_alg_groups}
		\ref{Prop.Preserving_alg_groups3},\ref{Prop.Preserving_alg_groups4} 
		applied to
		$\theta$ and by the characterization of the torus from Lemma~\ref{Lem.Characterize_Torus}, 
		$Y$ is a torus. Since $\dim T = n = \dim Y$ we get $Y \simeq T$ as varieties.
		
		\item $0 < \dim \theta(T)^\circ < n$: 
		As $\theta(\SL_n)$ normalizes $\theta(T)^\circ$, the $\theta(\SL_n(\ZZ))$-action
		by conjugation on $\theta(T)^\circ$ gives us a group homomorphism
		\[
		\eta \colon \theta(\SL_n(\ZZ)) \to \Aut_{\grp}(\theta(T)^\circ) = \GL_{\dim \theta(T)^\circ}(\ZZ) \, .
		\]
		The case $n = 1$ is impossible. If $n = 2$, then $\ker(\eta)$
		has finite index in $\theta(\SL_n(\ZZ))$ and if $n \geq 3$, by Margulis' normal 
		subgroup theorem~\cite[Theorem~8.1.2]{Zi1984Ergodic-theory-and}, 
		$\ker(\eta)$ is finite or has finite index in $\theta(\SL_n(\ZZ))$.
		We distinguish these two cases:
		\begin{enumerate}[label=\alph*)]
			\item $\ker(\eta)$ is finite: By~\eqref{Eq.Gamma_m_intersects_F_trivially},
			there exists a prime  $p$ such that $\ker(\eta)$ intersects $\theta(\Gamma_p)$ trivially.
			Hence, $\eta$ restricts to an injection $\theta(\Gamma_p) \to \Aut_{\grp}(\theta(T)^\circ)$.
			By Theorem~\ref{thm.CX}, we get $\dim \theta(T)^\circ \geq n$, contradiction.
			\item $\ker(\eta)$ has finite index in $\theta(\SL_n(\ZZ))$: 
			As $\theta^{-1}(\theta(T)^\circ)$ is uncountable (note that $\kk$ is uncountable), there exists 
			$t = (t_1, \ldots, t_n) \in \theta^{-1}(\theta(T)^\circ)$ and $i \in \{1, \ldots, n\}$ 
			such that $t_i \in \GG_m$ has infinite order.
			We may assume $i = n$.
			As $\theta^{-1}(\ker(\eta))$ has finite index in $\SL_n(\ZZ)$, there exists
			an integer $a \geq 1$ such that the elementary matrix 
			\[
			E_a \coloneqq 
			\begin{pmatrix}
				1 & \dots & a \\
				\vdots & \ddots &   \vdots  \\
				0 & \dots & 1
			\end{pmatrix}
			\in \SL_n(\ZZ)
			\]
			lies in $\theta^{-1}(\ker(\eta))$. As $t_n$ has infinite order,
			$E_a(t) = (t_1 t_n^a, t_2, \ldots, t_n) \neq t$.
			This contradicts the fact that $\theta^{-1}(\ker(\eta))$
			acts trivially on  $\theta^{-1}(\theta(T)^\circ)$ by conjugation.
		\end{enumerate}
		\item $\dim \theta(T)^\circ = 0$: Since $\theta(T)^\circ$ has countable index
		in $\theta(T)$, it follows that $\theta(T)$ is countable. However, since
		$\kk$ is uncountable, $T$ is uncountable as well (as $\dim T  \geq  1$) and thus
		we arrive at a contradiction. \qedhere
	\end{itemize}
\end{proof}

\section{Counterexamples in the projective case}
\label{sec.Counterexample_proj_case}

In this small section we give counterexamples to Theorem~\ref{mainthm.C} 
in the projective case. We assume that $\kk$ is uncountable.
The following easy remark will be useful:

\begin{remark}
	\label{rem.Negative_curves_are_T_invariant}
	Two distinct irreducible curves of negative self-intersection number in a smooth projective 
	surface cannot be linearly equivalent and thus there are at most countably many of them. As $\kk$ is uncountable,
	each irreducible curve of negative self-intersection number in a smooth projective toric surface 
	is invariant under the torus action.
\end{remark}

Let $X_0 \coloneqq \PP^2$ be endowed with the standard action of the torus $T = \GG_m^2$.
In this way, $X_0$ becomes a $T$-toric variety. The complement $C_0$ of the open $T$-orbit of $X_0$ consists of the three $T$-invariant lines in $\PP^2$. Let $\varphi_1 \colon X_1 \to X_0$ be the blow-up of the three $T$-fixed points, i.e.~the blow-up of the singular points of $C_0$. Then $X_1$ is again $T$-toric under the $T$-action that comes from $X_0$. The complement $C_1$ of the open $T$-orbit
in $X_1$ consists of six $-1$-curves, arranged in a hexagon and there are no other irreducible
curves of negative self-intersection number in $X_1$.

We will inductively construct $T$-toric smooth projective surfaces 
together with $T$-equivariant morphisms:
\[
\ldots \xrightarrow{\varphi_{i+2}} X_{i+1} \xrightarrow{\varphi_{i+1}}
X_i \xrightarrow{\varphi_{i-1}} \ldots
\xrightarrow{\varphi_{2}} X_1 \xrightarrow{\varphi_{1}} X_0 = \PP^2 \, .
\]
For $i \geq 1$, assume that $X_i$ is already constructed and 
that the complement $C_i$ of the open $T$-orbit in $X_i$ is connected and it is the union of all irreducible curves of negative self-intersection number. 
Moreover, the irreducible curves in $C_i$ are smooth.
Let $\varphi_{i+1} \colon X_{i+1} \to X_i$ be the blow-up of the singular points in $C_i$. Then $X_{i+1}$ is $T$-toric
under the $T$-action that comes from $\varphi_{i+1}$.
Moreover, the complement $C_{i+1}$ of the open $T$-orbit in $X_{i+1}$ is connected and it is the union of all 
irreducible curves of negative self-intersection number, by Remark~\ref{rem.Negative_curves_are_T_invariant}. 
Moreover, all irreducible curves in $C_{i+1}$ are smooth.
More precisely, the irreducible curves in $C_{i+1}$
that lie in the strict transform of $C_i$ under $\varphi_{i+1}$ have self-intersection number $<-1$ (as $C_i$ is connected), whereas the other irreducible curves in $C_{i+1}$ are the exceptional curves of $\varphi_{i+1}$. In particular, we get a group isomorphism
\[
\Aut(X_{i+1}) \to \Aut(X_i) \, , \quad \psi \mapsto \varphi_{i+1} \circ \psi \circ \varphi_{i+1}^{-1} \, .
\]

Hence, the smooth projective toric surfaces $X_i$, $i \geq 1$ are pairwise non-isomorphic (as the rank of the Picard groups of them are distinct), but the automorphism groups of them are all isomorphic to 
$\Aut(X_1) \simeq (\GG_m^2 \rtimes S_3) \rtimes (\ZZ / 2 \ZZ)$, where $S_3$ denotes the symmetric group of three elements,
see e.g.~\cite[\S3.3]{Be2007p-elementary-subgr}.

\begin{remark}
	In fact, ``almost all'' smooth projective toric varieties have an automorphism group, which is equal to 
	the maximal torus acting on it, see~e.g.~\cite[Corollary 4.7]{Co1995The-homogeneous-co}.
\end{remark}

\newcommand{\etalchar}[1]{$^{#1}$}
\providecommand{\bysame}{\leavevmode\hbox to3em{\hrulefill}\thinspace}
\providecommand{\MR}{\relax\ifhmode\unskip\space\fi MR }
\providecommand{\MRhref}[2]{%
	\href{http://www.ams.org/mathscinet-getitem?mr=#1}{#2}
}
\providecommand{\href}[2]{#2}


\begin{thebibliography}{AFK{\etalchar{+}}13}
	
	\bibitem[AA22]{ArAv2022Root-subgroups-on-}
	Ivan Arzhantsev and Roman Avdeev, \emph{Root subgroups on affine spherical
		varieties}, Selecta Math. (N.S.) \textbf{28} (2022), no.~3, Paper No. 60, 37.
	\MR{4414136}
	
	\bibitem[AFK{\etalchar{+}}13]{ArFlKa2013Flexible-varieties}
	I.~Arzhantsev, H.~Flenner, S.~Kaliman, F.~Kutzschebauch, and M.~Zaidenberg,
	\emph{Flexible varieties and automorphism groups}, Duke Math. J. \textbf{162}
	(2013), no.~4, 767--823.
	
	\bibitem[Bea07]{Be2007p-elementary-subgr}
	Arnaud Beauville, \emph{{$p$}-elementary subgroups of the {C}remona group}, J.
	Algebra \textbf{314} (2007), no.~2, 553--564. \MR{2344578}
	
	\bibitem[Bor91]{Bo1991Linear-algebraic-g}
	Armand Borel, \emph{Linear algebraic groups}, second ed., Graduate Texts in
	Mathematics, vol. 126, Springer-Verlag, New York, 1991.
	
	\bibitem[BP87]{BrPa1987Valuations-des-esp}
	Michel Brion and Franz Pauer, \emph{Valuations des espaces homog{\`e}nes
		sph{\'e}riques. ({Valuations} of spherical homogeneous spaces)}, Comment.
	Math. Helv. \textbf{62} (1987), 265--285 (French).
	
	\bibitem[Bri21]{Br2021Homogeneous-variet}
	Michel Brion, \emph{Homogeneous varieties under split solvable algebraic
		groups}, Indag. Math. (N.S.) \textbf{32} (2021), no.~5, 1139--1151.
	\MR{4310015}
	
	\bibitem[BT73]{BoTi1973Homomorphismes-abs}
	Armand Borel and Jacques Tits, \emph{Homomorphismes ``abstraits'' de groupes
		alg\'{e}briques simples}, Ann. of Math. (2) \textbf{97} (1973), 499--571.
	\MR{316587}
	
	\bibitem[Can14]{Ca2014Morphisms-between-}
	Serge Cantat, \emph{Morphisms between {C}remona groups, and characterization of
		rational varieties}, Compos. Math. \textbf{150} (2014), no.~7, 1107--1124.
	\MR{3230847}
	
	\bibitem[Cox95]{Co1995The-homogeneous-co}
	David~A. Cox, \emph{The homogeneous coordinate ring of a toric variety}, J.
	Algebr. Geom. \textbf{4} (1995), no.~1, 17--50 (English).
	
	\bibitem[CRX23]{CaReXi2023Families-of-commut}
	Serge Cantat, Andriy Regeta, and Junyi Xie, \emph{Families of commuting
		automorphisms, and a characterization of the affine space}, Amer. J. Math.
	\textbf{145} (2023), no.~2, 413--434. \MR{4570986}
	
	\bibitem[CX18]{CaXi2018Algebraic-actions-}
	Serge Cantat and Junyi Xie, \emph{Algebraic actions of discrete groups: the
		{$p$}-adic method}, Acta Math. \textbf{220} (2018), no.~2, 239--295.
	\MR{3849285}
	
	\bibitem[Dem82]{De1982Combinatorial-inva}
	A.~S. Demushkin, \emph{Combinatorial invariance of toric singularities.},
	Vestnik Moskov. Univ. Ser. I Mat. Mekh. (1982), no.~no. 2,, 80--87, 117.
	\MR{655409}
	
	\bibitem[DG11]{DeGr2011Schemas-en-groupes}
	Michel Demazure and Alexandre Grothendieck, \emph{Sch\'{e}mas en groupes ({SGA}
		3). {T}ome {III}. {S}tructure des sch\'{e}mas en groupes r\'{e}ductifs},
	vol.~8, Soci\'{e}t\'{e} Math\'{e}matique de France, Paris, 2011,
	S\'{e}minaire de G\'{e}om\'{e}trie Alg\'{e}brique du Bois Marie 1962--64.
	
	\bibitem[DL23]{DiLi2023On-the-Automorphis}
	Roberto D{\'\i}az and Alvaro Liendo, \emph{{On the Automorphism Group of
			Non-Necessarily Normal Affine Toric Varieties}}, Int. Math. Res. Not. IMRN
	(2023).
	
	\bibitem[Fil82]{Fi1982Isomorphisms-betwe}
	Richard~Patrick Filipkiewicz, \emph{Isomorphisms between diffeomorphism
		groups}, Ergodic Theory Dynamical Systems \textbf{2} (1982), no.~2, 159--171
	(1983). \MR{693972}
	
	\bibitem[FK18]{FuKr2018On-the-geometry-of}
	Jean-Philippe Furter and Hanspeter Kraft, \emph{On the geometry of the
		automorphism groups of affine varieties},
	\url{https://arxiv.org/abs/1809.04175}, 2018.
	
	\bibitem[Fre17]{Fr2017Algebraic-theory-o}
	Gene Freudenburg, \emph{Algebraic theory of locally nilpotent derivations},
	second ed., Encyclopaedia of Mathematical Sciences, vol. 136,
	Springer-Verlag, Berlin, 2017, Invariant Theory and Algebraic Transformation
	Groups, VII. \MR{3700208}
	
	\bibitem[Ful93]{Fu1993Introduction-to-to}
	William Fulton, \emph{Introduction to toric varieties}, Annals of Mathematics
	Studies, vol. 131, Princeton University Press, Princeton, NJ, 1993, The
	William H. Roever Lectures in Geometry. \MR{1234037}
	
	\bibitem[GP93]{GrPf1993Geometric-Quotient}
	Gert-Martin Greuel and Gerhard Pfister, \emph{{Geometric Quotients of Unipotent
			Group Actions}}, Proceedings of the London Mathematical Society
	\textbf{s3-67} (1993), no.~1, 75--105.
	
	\bibitem[GW20]{GoWe2020Algebraic-geometry}
	Ulrich G{\"o}rtz and Torsten Wedhorn, \emph{Algebraic geometry {I}. {Schemes}.
		{With} examples and exercises}, 2nd edition ed., Springer Stud. Math. --
	Master, Wiesbaden: Springer Spektrum, 2020 (English).
	
	\bibitem[Hum75]{Hu1975Linear-algebraic-g}
	James~E. Humphreys, \emph{Linear algebraic groups}, Springer-Verlag, New
	York-Heidelberg, 1975, Graduate Texts in Mathematics, No. 21.
	
	\bibitem[Kno91]{Kn1991The-Luna-Vust-theo}
	Friedrich Knop, \emph{The {L}una-{V}ust theory of spherical embeddings},
	Proceedings of the {H}yderabad {C}onference on {A}lgebraic {G}roups
	({H}yderabad, 1989), Manoj Prakashan, Madras, 1991, pp.~225--249.
	\MR{1131314}
	
	\bibitem[Kra17]{Kr2017Automorphism-group}
	Hanspeter Kraft, \emph{Automorphism groups of affine varieties and a
		characterization of affine {$n$}-space}, Trans. Moscow Math. Soc. \textbf{78}
	(2017), 171--186. \MR{3738084}
	
	\bibitem[KRvS21]{KrReSa2021Is-the-affine-spac}
	Hanspeter Kraft, Andriy Regeta, and Immanuel van Santen, \emph{Is the affine
		space determined by its automorphism group?}, Int. Math. Res. Not. IMRN
	(2021), no.~6, 4280--4300. \MR{4230395}
	
	\bibitem[Lie10]{Li2010Affine-Bbb-T-varie}
	Alvaro Liendo, \emph{Affine {$\Bbb T$}-varieties of complexity one and locally
		nilpotent derivations}, Transform. Groups \textbf{15} (2010), no.~2,
	389--425. \MR{2657447}
	
	\bibitem[Los09]{Lo2009Proof-of-the-Knop-}
	Ivan~V. Losev, \emph{Proof of the {K}nop conjecture}, Ann. Inst. Fourier
	(Grenoble) \textbf{59} (2009), no.~3, 1105--1134. \MR{2543664}
	
	\bibitem[LR22]{LeRe2022Vector-fields-and-}
	Matthias Leuenberger and Andriy Regeta, \emph{Vector fields and automorphism
		groups of {D}anielewski surfaces}, Int. Math. Res. Not. IMRN (2022), no.~6,
	4720--4752. \MR{4391900}
	
	\bibitem[LRU22]{LiReUr2022On-the-Characteriz}
	Alvaro Liendo, Andriy Regeta, and Christian Urech, \emph{On the
		characterization of {D}anielewski surfaces by their automorphism groups},
	Transform. Groups \textbf{27} (2022), no.~1, 181--187. \MR{4400720}
	
	\bibitem[LRU23]{LiReUr2023Characterization-o}
	\bysame, \emph{Characterization of affine surfaces with a torus action by their
		automorphism groups}, Ann. Sc. Norm. Super. Pisa Cl. Sci. (5) \textbf{24}
	(2023), no.~1, 249--289. \MR{4587746}
	
	\bibitem[Reg22]{Re2022Characterization-o}
	Andriy Regeta, \emph{Characterization of affine $\mathbb{G}_m$-surfaces of
		hyperbolic type}, \url{https://arxiv.org/abs/2202.10761}, 2022.
	
	\bibitem[RvS21]{ReSa2021Characterizing-smo}
	Andriy Regeta and Immanuel van Santen, \emph{Characterizing smooth affine
		spherical varieties via the automorphism group}, Journal de
	l{\textquoteright}\'Ecole polytechnique {\textemdash} Math\'ematiques
	\textbf{8} (2021), 379--414 (en). \MR{4218162}
	
	\bibitem[Ryb95]{Ry1995Isomorphisms-betwe}
	Tomasz Rybicki, \emph{Isomorphisms between groups of diffeomorphisms}, Proc.
	Amer. Math. Soc. \textbf{123} (1995), no.~1, 303--310. \MR{1233982}
	
	\bibitem[Ryb02]{Ry2002Isomorphisms-betwe}
	\bysame, \emph{Isomorphisms between groups of homeomorphisms}, Geom. Dedicata
	\textbf{93} (2002), 71--76. \MR{1934687}
	
	\bibitem[Ser03]{Se2003Permutation-group-}
	\'{A}kos Seress, \emph{Permutation group algorithms}, Cambridge Tracts in
	Mathematics, vol. 152, Cambridge University Press, Cambridge, 2003.
	\MR{1970241}
	
	\bibitem[Sha66]{Sh1966On-some-infinite-d}
	I.~R. Shafarevich, \emph{On some infinite-dimensional groups}, Rend. Mat. e
	Appl. (5) \textbf{25} (1966), no.~1-2, 208--212. \MR{0485898}
	
	\bibitem[Whi63]{Wh1963On-isomorphic-grou}
	James~V. Whittaker, \emph{On isomorphic groups and homeomorphic spaces}, Ann.
	of Math. (2) \textbf{78} (1963), 74--91. \MR{0150750}
	
	\bibitem[Zim84]{Zi1984Ergodic-theory-and}
	Robert~J. Zimmer, \emph{Ergodic theory and semisimple groups}, Monographs in
	Mathematics, vol.~81, Birkh\"{a}user Verlag, Basel, 1984. \MR{776417}
	
\end{thebibliography}
\end{document}